%% file: main.tex
\documentclass{article}
\input{preamble.tex}

\newcommand{\TOBY}[1]{}

\usepackage{authblk}
\title{\textbf{\papertitle}}
\author[1,2]{Sean Tull}
\author[2,3,4]{Johannes Kleiner}
\author[5,6]{Toby St Clere Smithe}

\affil[1]{Quantinuum}
\affil[2]{Association for Mathematical Consciousness Science}
\affil[3]{Ludwig Maximilian University of Munich}
\affil[4]{University of Bamberg}
\affil[5]{Topos Institute}
\affil[6]{VERSES Research}

\date{}

\begin{document}

\maketitle


\begin{abstract}
We present a categorical formulation of the cognitive frameworks of Predictive Processing (PP) and Active Inference, expressed in terms of string diagrams interpreted in a monoidal category with copying and discarding. This includes diagrammatic accounts of generative models, Bayesian updating, perception, planning, active inference, and free energy. In particular we present a diagrammatic derivation of the formula for active inference via free energy minimisation, and establish a compositionality property for free energy, allowing free energy to be applied at all levels of an agent's generative model. Aside from aiming to provide a helpful graphical language for those familiar with active inference, we conversely hope that this article may provide a concise formulation and introduction to the framework. 
\end{abstract}

\section{Introduction} \label{sec:intro}

\emph{Predictive processing} (PP) is a framework for modelling cognition and adaptive behaviour in both biological and artificial systems \cite{wiese2017vanilla,hohwy2020new}. A prominent sub-field is the programme of \emph{Active Inference}, developed by Friston and collaborators \cite{smith2022step,parr2022active,friston2017active,sajid2021active}, which aims to provide a unified understanding of cognition and action which can be applied at many levels, from a single neuron to an entire brain or organism. More specifically, active inference gives a proposal for how a cognitive agent represents its own beliefs about the world, how it updates these beliefs in light of new observations, and how it chooses the actions it takes, with the latter ultimately leading to new observations.



Central to the framework is that an agent possesses a \emph{generative model} which explains its observations causally in terms of both hidden states of the world and its own actions. Note that this model is internal to the agent, and typically distinct from the `true' causal process in the world which produces the observations. After receiving an observation, the agent may \emph{update} this generative model to determine likely hidden states which caused the observation (the process of perception) and choose its actions (the process of planning). In active inference, both forms of updating are carried out together through a form of approximate Bayesian inference, by minimising a quantity known as \emph{free energy} \cite{friston2006free,friston2010free}.

While active inference seeks a principled account of cognition, at present its formalisation can seem fairly complex, and there are various aspects which do not follow immediately from simply applying the definitions to a given generative model. Conceptually clear formal accounts of the framework would be desirable to simplify the theory and address these issues, as well as for applications within AI. 

One hope for such a formal account would be for it to be both \emph{compositional} and \emph{graphical}. Indeed the generative models in PP are highly structured, often given as `hierarchical models' \cite{de2017factor} which are best represented diagrammatically in terms of probabilistic graphical models such as Bayesian networks. 
While there has been support for, and steps towards, a graphical account of active inference \cite{friston2017graphical}, so far the graphical aspects only formally describe the structure of a generative model, while other aspects such as updating and free energy are still treated through traditional probabilistic calculations, and only informally in diagrams. 

Recently however, fully formal diagrammatic methods have been developed for both describing Bayesian networks and carrying out probabilistic reasoning about them. These approaches are based on (monoidal) \emph{category theory} and its associated graphical language of \emph{string diagrams} \cite{piedeleu2023introduction}. Category theory has been applied across the sciences as a general mathematics of interacting \emph{processes}, including within probability theory \cite{coecke2012picturing,cho2019disintegration}, causality \cite{jacobs2019causal,fritz2023d,lorenz2023causal}, game theory \cite{ghani2018compositional}, machine learning \cite{fong2019backprop,shiebler2021category}, quantum computing \cite{abramsky2004categorical} and natural language processing \cite{clark2008compositional}. In particular a major ongoing development is in the study of probabilistic processes in terms of \emph{cd-categories} (and `Markov categories'), which allow one to carry out probabilistic reasoning entirely through string diagrams \cite{fritz2020synthetic}. 

In this work we give a full categorical account of predictive processing and active inference in terms of string diagrams, interpreted in cd-categories. In doing so we aim to give a conceptually clear account of the main features of the framework: generative models, Bayesian updating (including with soft observations), perception, action planning, and their combination in active inference, and both \emph{variational} and \emph{expected} free energy. 

A highlight is a fully graphical derivation of the well-known formula for active inference in terms of minimisation of free energy. While this is a central result within active inference, its usual justification is more heuristic in nature. Here we instead derive the free energy formula purely graphically from a diagrammatic account of active inference itself, providing what we argue is the most transparent account of this result known so far. 

The categorical perspective also naturally leads us to consider more novel aspects of active inference. These include the definition of \emph{open generative models} (essentially from \cite{lorenz2023causal}) which are generative models coming with `inputs', allowing them to serve as the  building blocks of an overall generative model. 

We also introduce a notion of variational free energy for open models which allows us to establish the desirable property that free energy is \emph{compositional}. Namely, a system with an overall generative model composed from sub-models may minimise global VFE by minimising VFE locally within each component. This is a crucial fact in order to apply free energy as proposed to all levels of a system, say from a whole brain down to its individual neurons. 

Overall, we hope that our diagrammatic accounts of PP can provide a conceptually clear view of the framework, and also a natural language for reasoning within it. Indeed, as argued for example in \cite{lorenz2023causal} and elsewhere \cite{jacobs2019causal} diagrams in cd-categories provide a natural way to both represent causal (generative) models, as well as reason about them. As we demonstrate here they are also natural for describing the structure of active inference, including free energy. Aside from aiming to provide a helpful graphical language for those familiar with active inference, we conversely hope that this article may provide a succinct introduction to PP for those already familiar with string diagrams and categorical reasoning.  

\para{Further motivations}

Though primarily a framework for cognition, various proposals have been put forward for how predictive processing may be related to \emph{consciousness} \cite{wiese2017vanilla}. In previous work, two of the authors developed a categorical account of the \emph{Integrated Information Theory} of consciousness, again essentially using cd-categories \cite{TullKleiner,kleiner2021mathematical} and based on the work here we hope to give a categorical account of how consciousness may be accounted for within PP \cite{deane2021consciousness,hohwy2020predictive}. We also see this work as a piece of the programme of \emph{Compositional Intelligence}, which explores how categorically structured models and processes can be applied to (artificial) intelligence. Specifically, PP may be seen as a proposal for how 
compositional intelligence manifests in biology; that is, how biological systems may employ compositionality to carry out intelligent and adaptive behaviour.   

Active inference can also be understood as an alternate proposal to reinforcement learning (RL) for how agents can learn adaptive behaviour, and shares similar features including the role of probabliistic models and inference \cite{tschantz2020reinforcement}. It differs from conventional RL by replacing an explicit reward function with the aim of maximizing evidence for a probabilistic model, where the agent's preferences are now encoded in the model's prior distribution \cite{friston2009reinforcement}.

\para{Related work}

This work can be seen as a part of the growing field of `categorical cybernetics' \cite{smithe2021cyber,capucci2021towards}, including previous work from one of the authors on compositional accounts of Bayesian updating \cite{smithe2020bayesian} and of active inference in terms of `statistical games' \cite{smithe2021compositional,smithe2022compositional}. It differs from previous works by directly formalising the active inference framework itself, and by working explicitly graphically within the simple string-diagrammatic setting of cd-categories, with the aim of supplying a simple abstract characterization of active inference agents.

In this way the work is a part of a general movement in applying string diagrams in cd-categories to probability theory and causal reasoning. A categorical account of Bayesian inversion was first given by Coecke and Spekkens in \cite{coecke2012picturing}, and then within cd-categories by Cho and Jacobs \cite{cho2019disintegration}, with further developments in categorical probability by Fritz \cite{fritz2020synthetic}. Our diagrammatic account of generative models is precisely that given for causal models in part by one of the authors in \cite{lorenz2023causal}, which builds on the earlier categorical treatments of (causal) Bayesian networks by Fong \cite{fong2013causal}, Jacobs et al \cite{jacobs2019causal} and others e.g. \cite{fritz2023d}. Indeed, as an agent's explanation for the  observations it receives from the world, a generative model is ultimately a causal model \cite{pearl2009causality}, though this is not often stressed in the literature.  

The two forms of soft Bayesian updating treated here we first studied by Jacobs in \cite{jacobs2019mathematics}. The specific treatment of conditioning in cd-categories used here is from \cite{lorenz2023causal}. Cd-categories with (non-unique) conditionals have also been recently studied as `partial Markov categories' in \cite{di2023evidential}, along with both notions of updating. Our treatment of free energy refers to the KL divergence of distributions; we note that an axiomatic treatment of Markov categories coming with divergences on their morphisms has recently been given by Perrone in \cite{perrone2022markov}.

Within active inference itself, graphical aspects have been increasingly prominent, with discussion of the `graphical brain' in \cite{friston2017graphical}. In such works it is argued that one may describe models as (non-directed) Forney Factor Graphs (FFGs) \cite{de2017factor}. However, generative models are inherently directed, going from states to observations with the other direction being intractable to compute exactly. Thus it is more natural to treat models using (generalisations of) directed Bayesian networks. Nonetheless we note though that FFGs derived from a model still have a role when minimising VFE via `message passing' algorithms \cite{parr2019neuronal}.

\begin{figure} 
\[
\includegraphics[scale=0.12]{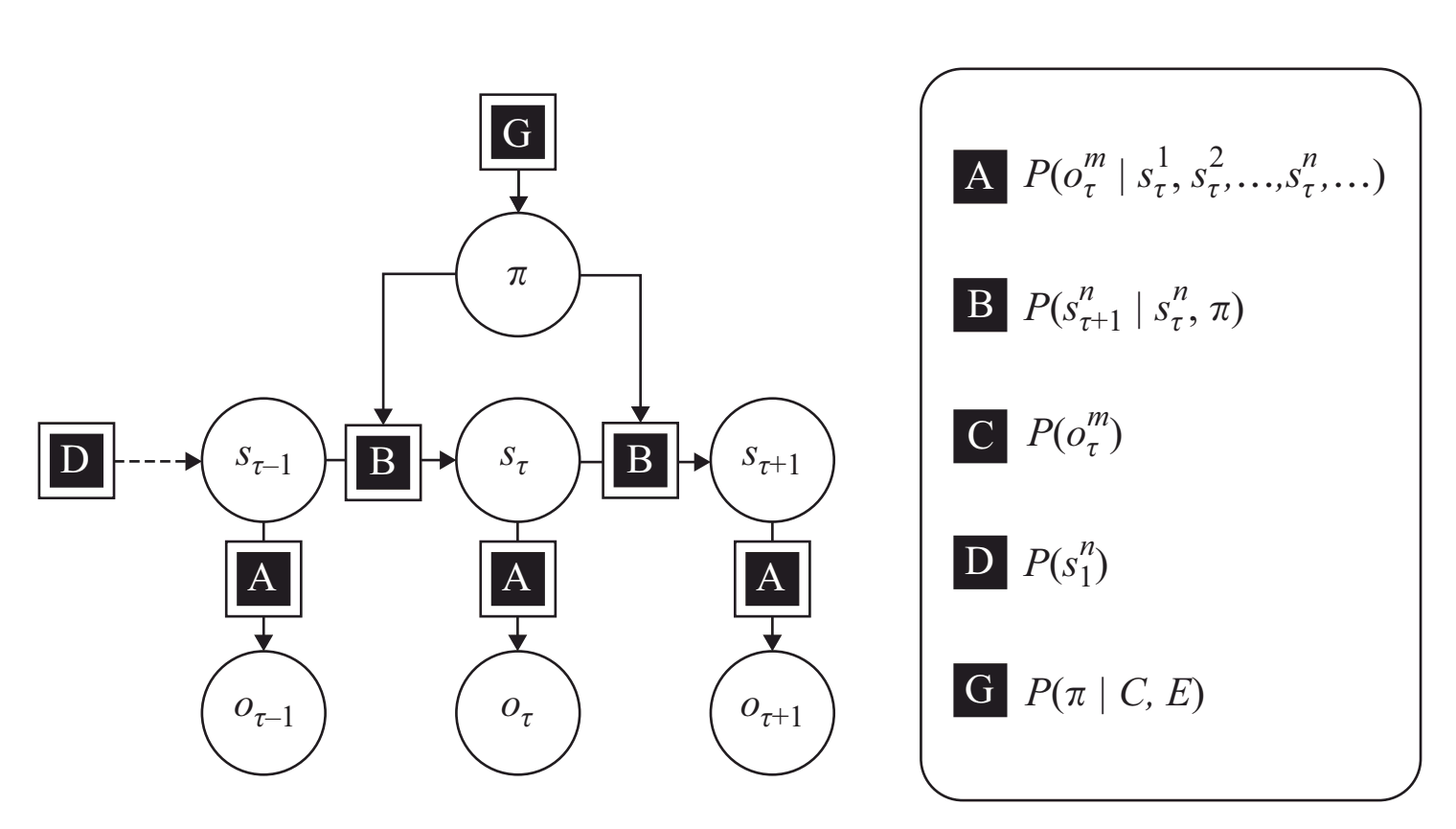}
\qquad 
\qquad 
\qquad
\input{sdiagex.tikz}
\]
\begin{caption} \ 
A generative model diagram from the recent book \emph{Active Inference} by Parr, Pezzulo and Friston \cite{parr2022active} (left) and the equivalent string diagram representation (right) (though replacing the informal EFE term $\EFE$ with the prior $E$).
\end{caption}
\label{Fig:diagcomparison}
\end{figure}

Interestingly, while a Bayesian network is typically depicted as a DAG (with only the variables labelled), one may argue the diagrams in active inference have been naturally `converging' on their string diagrammatic representation, which also includes labels on the channels; see Figure 1. We claim that the advantage of string diagrams beyond DAGs is in allowing one to both represent and reason about the model in the same formalism.  

We note that the diagrams in PP are at times only semi-formal, including aspects such as the free energy which are not strictly part of the generative model. One may see this work as a step towards the shared goal of representing formally all aspects of PP within one language of diagrams.

\para{Structure of the article}

We begin in Section \ref{sec:setup} by introducing cd-categories and their diagrammatic account of probability theory. We then apply these to introduce from scratch the key aspects of PP: generative models as Bayesian networks, and their generalisation to open generative models in a cd-category (Section \ref{sec:gen-models}), (Bayesian) updating of generative models from  observations (Section \ref{sec:updating} ), perception and planning (Section \ref{sec:percep-planning}) and their combination in exact active inference (Section \ref{sec:exact-act-inf}). We then discuss free energy (Section \ref{sec:free-energy}) and give a graphical derivation of active inference via free energy minimisation (Section \ref{sec:approx-actinf}). In Section \ref{sec:comp-FE} we then introduce free energy for open models using a graphical formalism of `log-boxes' and use this to establish the compositionality property of free energy. Finally we discuss future work in Section \ref{sec:futurework}.

\para{Acknowledgements} 
We thank Robin Lorenz for helpful discussions and development of the treatment of causal models used here.
This research was supported by grant number FQXi-RFP-CPW-2018 from the Foundational Questions Institute and Fetzer Franklin Fund, a donor advised fund of the Silicon Valley Community Foundation. Sean Tull would also like to thank Quantinuum for their generous support in this research. 
Johannes Kleiner would like to thank the Mathematical Institute of the University of Oxford for hosting him while working on this research.

\section{Categorical Setup} \label{sec:setup}

Let us begin by introducing the graphical treatment of probabilistic processes in terms of \emph{string diagrams}, developed by numerous authors \cite{coecke2012picturing,cho2019disintegration,fritz2020synthetic}. Formally, these correspond to working in a `monoidal category' or more specifically a `cd-category', but in practice one may avoid mathematical details and simply work with the diagrams themselves. Though cd-categories are very general, in this article it suffices to consider the category $\MatR$ of $\mathbb{R}^+$ valued finite matrices, introduced shortly in Example \ref{ex:MatR+}. 

A \emph{category} $\catC$ consists of a collection of \emph{objects} $X, Y, \dots$ and \emph{morphisms} or \emph{processes} $f \colon X \to Y$ between them, which we can compose in sequence. In string diagrams we depict an object $X$ as a wire and a morphism $f \colon X \to Y$ as a box with lower input wire $X$ and upper output wire $Y$, read from bottom to top. 
\[
\input{box.tikz} 
\]
Given another morphism $g \colon Y \to Z$ we can compose them to yield a morphism $g \circ f \colon X \to Z$, depicted as:
\[
\input{composite-1.tikz}
\]
Each object $X$ also comes with an \emph{identity} morphism $\id{X} \colon X \to X$ depicted as a blank wire:
\[
\input{identity.tikz}
\]
The identity leaves any morphism alone under composition, that is $\id{Y} \circ f = f = f \circ \id{X}$ for any $f \colon X \to Y$. 

Formally, a \emph{symmetric monoidal category} $(\catC, \otimes, I)$ is a category $\catC$ with a functor $\otimes \colon \catC \times \catC \to \catC$, and natural transformations which express that $\otimes$ is suitably associative and symmetric, with a distinguished unit object $\monunit$  \cite{coecke2006introducing}. All of these aspects however are expressed most simply in diagrams.

Firstly, the \emph{tensor} $\otimes$ allows us to compose any pair of objects $X, Y$ into an object $X \otimes Y$, depicted by placing wires side-by-side. 
\[
\input{tensor-ob.tikz}
\]
Given morphisms $f \colon X \to W$ and $g \colon Y \to Z$ we can similarly form their `parallel composite' $f \otimes g \colon X \otimes Y \to W \otimes Z$ as below. 
\[
\input{tensor.tikz}
\]
In text we will at times omit the tensor symbols and write e.g. `$f$ from $X$ to $Y, Z$' or $f \colon X \to Y, Z$ in place of $f \colon X \to Y \otimes Z$. 

The tensor is symmetric so we can `swap' pairs of wires past each other, such that swapping twice returns the identity, and boxes carry along the swaps as below.
\[
\input{symmetry.tikz}
\]
We also have a distinguished \emph{unit object} $\monunit$ whose identity morphism we depict simply as empty space, and denote by $1$. 
\begin{equation} \label{eq:empty-space}
\input{idI.tikz}  \ \ \quad =   \ \ \quad 1 
\end{equation} 
Intuitively, tensoring any object by the unit simply leaves it invariant. The unit allows us to consider morphisms with `no inputs' and/or `no outputs' in diagrams. A morphism $\omega \colon \monunit \to X$ is called a \emph{state} of $X$, depicted with no input. An \emph{effect} on $X$ is a morphism $e \colon X \to \monunit$, depicted with no output. A morphism $r \colon \monunit \to \monunit$, drawn with no input or output, is called a \emph{scalar}. 
\[
\input{stateomega.tikz} \qquad \input{effect.tikz}  \qquad \input{scalar.tikz}
\]
In particular the `empty space' diagram \eqref{eq:empty-space} is the scalar $1 = \id{\monunit}$.  

The compositions $\circ, \otimes$ satisfy axioms which must be considered when working symbolically but are trivial in the graphical language. An example is associativity of composition $(h \circ g) \circ f = h \circ (g \circ f)$, which is automatic from simply drawing three boxes in sequence on the same wire. Similarly the rule $(f \otimes g) \circ (f' \otimes g') = (f \circ f') \otimes (g \circ g')$, displayed in the left identity below, and the `interchange law' $(f \otimes \id{})\circ(\id{}\otimes g) = f \otimes g = (\id{} \otimes g)\circ(f\otimes\id{})$, displayed in the right identity below, amount to letting us freely slide boxes along wires. 
\[
\input{functoriality2.tikz} \qquad \qquad \qquad 
\input{interchange2.tikz} 
\]
Let us now introduce our primary example category in this article. 

\begin{example} \label{ex:MatR+}
In the category $\MatR$ of positive valued matrices, the objects are finite sets $X,Y,\dots$ and the morphisms $M \colon X \to Y$ are functions $M \colon X \times Y \to \mathbb{R}^+$ where $\mathbb{R}^+ := \{r \in \mathbb{R} \mid r \geq 0 \}$. Equivalently such a function is given by an `$X \times Y$ matrix' with entries $M(y \mid x) := M(x,y) \in \mathbb{R}^+$ for $x \in X$, $y \in Y$.  
\[
\input{mat-mor.tikz} 
\quad   :: \   (x,y) \  \mapsto \ M(y \mid x)
\]
Composition of $M \colon X \to Y$ and $N \colon Y \to Z$ is given by summation over $Y$: 
\[
\input{mat-comp.tikz}
\quad   :: \   (x,z) \  \mapsto \ \sum_{y \in Y} N(z \mid y) M(y \mid x)
\]
The tensor $\otimes$ is given on objects by the Cartesian product $X \otimes Y = X \times Y$, and on morphisms by the Kronecker product, i.e. the usual tensor product of matrices:
\[
\input{mat-tens.tikz} 
\quad :: \  ((x,y),(w,z)) \  \mapsto \ M(w \mid x)N(z \mid y)
\]
The symmetry is simply the isomorphism $X \times Y \simeq Y \times X$. The unit object $\monunit=\{\star\}$ is the singleton set. A state of $X$ is then equivalent to a positive function on $X$:
\[
\input{state.tikz}  \quad   :: \    x \  \mapsto \ \omega(x)
\]
where $\omega(x) := \omega(x \mid \star)$. Similarly, an effect $e$ on $X$ is also equivalent to a positive function on $X$ via $e(x) := e(\star \mid x)$.
\[
\input{effect.tikz}  \quad   :: \    x \  \mapsto \ e(x)
\]
Finally, a scalar is precisely a positive real $r \in \mathbb{R}^+$.
\end{example} 

\subsection{Cd-categories} 

Many aspects of probability theory can be treated entirely diagrammatically, by noting that categories such as $\MatR$ come with the following further structure. 

\begin{definition} \cite{cho2019disintegration} \label{def:cd_category}
A \emph{cd-category} (\emph{copy-discard category}) is a symmetric monoidal category in which each object comes with a specified pair of morphisms
\[
\input{copy-delete.tikz}
\]
called \emph{copying} and \emph{discarding}, respectively, which satisfy the following:
\[
\input{markov-axioms.tikz}
\]
The choice of these morphisms is moreover `natural' in that the following hold for all objects $X, Y$.
\begin{equation} \label{eq:nat-rules}
\input{copy-nat.tikz} \qquad \qquad \input{disc-nat.tikz} \qquad \qquad \input{disc-I.tikz}
\end{equation}
\end{definition}

Thanks to these axioms for copying, we can unambiguously define a copying morphism with $n$ output legs, for any $n \geq 1$, via: 
\[
\input{copyn.tikz}
\]
with the $n=0$ case defined to be discarding $\discard{}$. 

The presence of discarding allows us to identify the truly `probabilistic' processes in a cd-category. We say that a morphism $f$ is a \emph{channel} when it preserves discarding, as below. 
\[
\input{causal2.tikz}
\]
In particular, we call a state $\omega$ \emph{normalised} when the following holds. For an explanation of why this gives the usual definition, cf. Example~\ref{ex:MatR+cont} below.
\[
\input{state-norm.tikz}
\]
Here we will often call a normalised state $\omega$ of $X$ a \emph{distribution} of $X$, even when working in a general cd-category \footnote{This is to avoid confusion with the usual use of the term (hidden) `state' in PP.}. We also call a normalised state of $X \otimes Y$ a \emph{joint distribution} over $X, Y$.

A cd-category in which every morphism is a channel, or equivalently $\discard{}$ is the unique effect on any object, is called a \emph{Markov category} \cite{fritz2020synthetic}. Given any cd-category $\catC$, its subcategory $\catC_\channel$ of channels always forms a Markov category. 

Discarding allows us to `ignore' certain outputs of a process. Given any morphism $f$ from $X$ to $Y, Z$, its \emph{marginal} $X \to Y$ is the following morphism:  
\[
\input{marginalmor.tikz} 
\]
Let us see how these features describe discrete probability theory within our example category.

\begin{example}\label{ex:MatR+cont}
$\MatR$ is a cd-category. Copying on $X$ is given $\tinymultflip[whitedot](y,z \mid x) = \delta_{x,y,z}$ with value $1$ iff $x=y=z$ and $0$ otherwise. Discarding $\discard{}$ on $X$ is given by the function with $x \mapsto 1$ for all $x \in X$. Hence a state $\omega$ is normalised, i.e. forms a distribution on $X$, precisely when it forms a \emph{probability distribution} over $X$ in the usual sense, i.e.~its values sum to $1$. 
\[
 \input{state-sum.tikz}
 \ \ = \ \  
\sum_{x \in X} \omega(x)  \ \ = \  \ 1
\]

More generally, a process $M \colon X \to Y$ is a channel iff it forms a \emph{probability channel}, or equivalently a \emph{stochastic} matrix, meaning that it sends each $x \in X$ to a normalised distribution $M(y \mid x)$ over $Y$.  Indeed we have that: 
\[
\input{m-disc.tikz}  \ \ ::  \ \ x  \ \mapsto \ \sum_y M(y \mid x)
\]
Hence $M$ is a channel iff this effect is constant at $1$, i.e. for all $x$ we have \[\sum_{y \in Y} M(y \mid x) = 1\]

In typical probability theory, such a channel is also often called a `conditional probability distribution' $P(Y \mid X)$ with values denoted $P(y \mid x) := P(Y=y \mid X =x)$ for $x \in X$, $y \in Y$.  The subcategory of channels in $\MatR$ is the Markov category $\FStoch$ of finite Stochastic matrices. 

Let us see how a few features of probability theory appear in diagrams. Firstly, for any $X, Y$, a distribution $\omega$ on $X \otimes Y$ corresponds to a joint distribution over $X, Y$ (left-hand below). In particular given a pair of distributions $\phi, \sigma$ over $X, Y$, the distribution $\phi \otimes \sigma$ corresponds to the resulting product distribution over $X \times Y$, with $X$ and $Y$ independent from each-other (right-hand below).
\[
\input{joint-state.tikz}
\qquad \qquad \qquad 
\input{prod-state.tikz}
\]
A general channel as below represents a probability channel \rl{$P(Y_1,\dots,Y_m \mid X_1,\dots,X_n)$. }
\[
 \input{MXY.tikz} 
\]
Marginalisation of any morphism corresponds to the usual notion in probability theory, given by summation over the discarded object.
\[
\input{marginal2.tikz} :: x \mapsto \sum_{y \in Y} \omega(x,y)
\qquad \qquad \qquad 
\input{margM.tikz} :: (x,y) \mapsto \sum_{z \in Z}  M(y,z \mid x)
\]
Finally we observe that for any effect $e \colon X \to \mathbb{R}^+$ and distribution $\omega$ the scalar $e \circ \omega$ corresponds to the expectation value of the function $e$ according to the probability distribution $\omega$. 
\[
\mathop{\mathbb{E}}_{x \sim \omega} e(x)  \ \ = \ \   \input{exp-value.tikz}  \ \ = \ \ \sum_{x \in X} e(x) \omega(x)
\]
\end{example}

\subsection{Sharp states and caps} 

The copying morphisms in a cd-category allow us to identify those states which are really `deterministic' \cite{fritz2020synthetic}. We call a state $x$ \emph{sharp}, and depict it with a triangle as below, when it is copied by $\tinycopy$, that is:
\begin{equation} \label{eq:copy-points}
\input{copy-points.tikz}
\end{equation}
In many categories there is a corresponding effect for each state, playing an important role for sharp states, thanks to the following feature. We say that $\catC$ has \emph{caps} when each object comes with a distinguished effect on $X \otimes X$ depicted $\tinycap$ and satisfying: 
\begin{equation}\label{eq:mult-map}
\input{cap-sym.tikz} \qquad \qquad \ \ \ 
\input{cap-2.tikz}  \ \ \ \qquad \qquad  \input{mult-map2.tikz} 
\end{equation}
and such that the following holds for all objects $X, Y$:
\[
\input{capXYrule.tikz} 
\]
Intuitively, the cap is an effect which checks if its two input wires are `in the same state'. The first equation in \eqref{eq:mult-map} expresses that this comparison is symmetric, and the remaining two that it is compatible with copying; for example the second says that each input when copied is equal to its copy. 

Practically, caps allow us to `turn outputs into inputs'. In particular, for each state $\omega$ we can define a corresponding effect by `flipping $\omega$ upside-down':
\begin{equation} \label{eq:eff-from-state}
\input{cap-4.tikz}
\end{equation}
When $\omega =x $ is a sharp state, we call this effect \emph{sharp} also. One may verify that it is the unique effect satisfying the following. 
\begin{equation} \label{eq:sharp-state-eff}
\input{sharp-1.tikz} \qquad \qquad 
\input{sharp-2.tikz}
\end{equation}
Caps are particularly useful in diagrammatic reasoning when they are \emph{cancellative}, meaning that: 
\[ 
\input{capcancel.tikz} \label{eq:cancel-caps}
\]
for all morphisms $f, g$.

\begin{example}
$\MatR$ has cancellative caps. Each point $x \in X$ corresponds to a normalised sharp state on $X$ which we again denote by $x$, given by the point probability distribution $\delta_x$ at $X$. 
\[
\input{sharp-state-x.tikz} :: \ y \mapsto 
\begin{cases} 
1 & x=y \\ 
0 & \text{otherwise} 
\end{cases} 
\]
The corresponding effect is given by the function $\delta_x$ also. Each cap is given by \ $\tinycap(x,y) = \delta_{x,y}$. We note a useful fact that for any morphism $M \colon X \to Y$ its values $M(y \mid x)$ can be given diagrammatically as below. 
\[
\ M(y \mid x)  \ \ \ =  \ \ \ \input{mat-values.tikz}  
\]
Every sharp state on $X$ is of the above form for some $x \in X$, or else given by the \emph{zero state} $0$ defined by $0(x) = 0$ for all $x \in X$. The only sharp scalars are $0$ and $1$. Note that a general state $\omega$, even when normalised, is not copyable. 
\[
\input{noncopy.tikz}
\] 
Indeed the left-hand side is the distribution $(x,y) \mapsto \omega(x) \delta_{x,y}$, while the right is $(x,y) \mapsto \omega(x) \omega(y)$, which differ unless $\omega$ is zero or $\omega = \delta_x$ for some $x \in X$. 
\end{example}

\subsection{Normalisation} 

In graphical probabilistic reasoning it is also useful to be able to normalise states and processes. We say that a cd-category $\catC$ has \emph{normalisation} when it comes with a rule assigning each morphism $f \colon X \to Y$ a new morphism called the \emph{normalisation} of $f$, depicted by drawing a dashed blue box:
\begin{equation} \label{eq:norm-box}
\input{normmap.tikz}
\end{equation}
such that these normalisations satisfy various axioms, of which we sketch a few here. For a full definition see \cite{lorenz2023causal}. Firstly, a general state $\omega$ is equal to a scalar multiple of its normalisation. In particular in $\MatR$ when the state is non-zero, this means that its normalisation is indeed normalised in our above sense, i.e. a distribution. 
\begin{equation} \label{eq:state-norm}
\input{state-norm-cond.tikz}
\end{equation}
For a general morphism $f$ its normalisation is given on each sharp state $x$ by normalising $f \circ x$.
\begin{equation} \label{eq:norm-on-states}
\input{norm-det.tikz}
\end{equation}
These two rules combine to give the following equation without explicit reference to states. 
\begin{equation} \label{eq:norm-sup-cond}
\input{normsupcond1.tikz}
\end{equation}
Note that if $f$ we already a channel then it would be equal to its normalisation, as in this case we can passing the discarding through $f$ and then the copy map above. In general normalisations satisfy a few graphical conditions including the following.
\begin{equation} \label{eq:norm-axioms} 
\input{norm-monoidal.tikz}
\qquad \qquad 
\input{copynorm.tikz}
\end{equation} 
Further, for all morphisms $f$ and channels $g$ we have:
\begin{equation} \label{eq:norm-axioms-more}
\input{normchan3.tikz}
\end{equation}
and for all sharp states $x$ and morphisms $f$ we have the following. 
\[
 \input{norm-det-wide.tikz}
\]

For a full account of the properties of normalisation see \cite{lorenz2023causal}. We note that for a general morphism $f$, its normalisation is not necessarily a channel but only a `partial channel' \footnote{Such morphisms are called 'quasi-total' in \cite{di2023evidential}, where morphisms satisfying \eqref{eq:norm-sup-cond} are also called `normalisations', though are not uniquely chosen unlike our definition.}. In terms of states, this is because its sends each sharp state $x$ either to a normalised state, or else to $0$ if $f \circ x = 0$. However in $\MatR$ it will be a channel provided $f$ has `full support', so that $f \circ x$ is non-zero for all non-zero sharp states $x$. 

 Throughout the article, the following notation will be useful. For any set $X$ and function $f \colon X \to \mathbb{R}^+$ let us write 
\[
\Norm{x} f(x) := \frac{f(x)}{\sum_{x' \in X}f(x')}
\] 
whenever this is well-defined, i.e. the denominator is finite and non-zero.

\begin{example}
$\MatR$ has normalisation. On each object $X$ the zero state $0$, given by $0(x) = 0$ for all $x \in X$, is defined to have normalisation $0$. For any non-zero state $\omega$ we indeed have 
\[
\input{normomega.tikz}  \quad   :: \    x \  \mapsto \Norm{x} \omega(x)
\]
For a general morphism $M \colon X \to Y$ the normalisation is given by:
\[
\input{normmat.tikz} \quad   :: \  (x, y) \mapsto  \ 
\begin{cases} 
\Norm{y} M(y \mid x)  & \text{ if }
\sum_{y \in Y} M(y \mid x) \neq 0  \\ 
0 & \text{otherwise}
\end{cases}
\]
As a result if $M$ has full support, so that $M(y \mid x) \neq 0$ for some $y$, for all $x$, then its normalisation is a probability channel. 
\end{example}

\subsection{Further cd-categories}

Though we will not need them here, we note that the notion of a cd-category is much more general than $\MatR$, and give a few examples for those familiar with them. The category $\Rel$ whose objects are sets and morphisms are relations is a cd-category, as are its subcategories $\cat{PFun}$ of sets and partial functions and $\cat{Set}$ of sets and functions, with the latter forming the channels in $\cat{PFun}$.

There are also many more cd-categories of a `probabilistic' nature, see for instance \cite{panangaden1998probabilistic,cho2019disintegration,fritz2020synthetic}. In particular to treat general probability spaces (including `continuous probability channels') one may work in the category $\Kl(G)$ of measurable spaces $X=(X,\Sigma_X)$ and Markov (sub-)kernels $f \colon X \to Y$, which send each $x \in X$ to a (sub-)probability measure $f(x)$ over $Y$. Roughly, this means replacing all instances of summation $\Sigma$ in $\MatR$ above with integration $\int$. Of particular interest in PP is the following subcategory, though we will not work with it in detail in this article.

\begin{example} \cite[Section 6]{fritz2020synthetic}
In the category $\Gauss$ the objects are spaces $X = \mathbb{R}^n$ and morphisms $M \colon X \to Y$ are Markov kernels $f \colon X \to Y$ with densities of the form $f(y \mid x) =\eta(y - Mx)$ for some fixed Gaussian noise distribution $\eta$ (independent of $x$) and linear map $M \colon X \to Y$. This category models linear processes with Gaussian noise. More general non-linear Gaussian processes are studied in PP under the so-called `Laplace assumption'. 
\end{example}

\section{Generative Models} \label{sec:gen-models}

A central feature in PP is that each cognitive agent possesses a generative model which describes their internal beliefs about how the observations they receives arise from hidden states of the world\footnote{Note that is distinct from whatever `true' external process produces the observations in reality, with the latter often called the `generative process' to distinguish it from the agent's own `generative model' \cite{parr2022active}.}. In its simplest form, a generative model consists of a channel $c \colon S \to O$ describing how likely $c(o \mid s)$ a given observation $o \in O$ is for each hidden state $s \in S$, along with a distribution $\sigma$ over $S$ describing prior beliefs about how likely each state is. 


However, generative models typically come with further compositional structure, relating various spaces of observations and hidden states, as formalised by a \emph{Bayesian network} (or more precisely a \emph{causal} Bayesian network, see later discussion), a probabilistic graphical model based on a directed acyclic graph (DAG). There is in fact a close correspondence between DAGs and cd-categories, allowing us to describe and study such models entirely in terms of string diagrams. This view also leads one to consider more general `open generative models', which may come with `input' variables. These open models can be used to which describe the individual components of an overall generative model in the usual sense. For more details on the approach used here, see \cite{lorenz2023causal}.



We begin by relating DAGs with the following class of string diagrams. 

\begin{definition} \label{Def_NetworkDiagram}
\cite{lorenz2023causal}
A \emph{network diagram} is a string diagram $D$ built from single-output boxes, copy maps and discarding:
\[
\input{nd-box.tikz}
\qquad 
\input{nd-copy.tikz}
\qquad 
\input{nd-disc.tikz}
\] 
with labellings on the wires, such that any wires not connected by a sequence of copy maps are given distinct labels, and each label appears as an output at most once and as an input to any given box at most once. 
\end{definition}

 Such diagrams are best understood by examples, which we come to shortly. Before this, we note that network diagrams are in fact equivalent to DAGs in the following sense. By an \emph{open DAG} we mean a finite DAG $G$ with vertices $V=\{X_1,\dots,X_n\}$, along with subsets $I, O \subseteq V$ of \emph{input} and \emph{output} vertices, respectively, such that each input vertex has no parents in $G$. 

Given any open DAG $G$, we may construct an equivalent network diagram featuring a box $c_i$ with output $X_i$ for each non-input vertex $X_i$. The box $c_i$ itself has an input wire for each parent of $X_i$ in $G$. In the diagram we copy the output of this box and pass it to each of the children of $X_i$, as well as an extra time if $X_i \in O$ i.e. $X_i$ is an output vertex of the DAG. 
\[
\input{DAG-extract.tikz}
\qquad 
\mapsto
\qquad 
\input{DAG-extractsd.tikz}
\]

By construction, this yields a network diagram $D_G$ from the inputs $I$ to the outputs $O$ of the DAG. Conversely, given any such network diagram $D$ we define an open DAG $G_D = (G,I,O)$ with a vertex $X \in V$ for each wire $X$ in $D$, and with $X \in I, O$ iff it is an input (resp. output) to the diagram.

In practice the labellings of the boxes are arbitrary, and we consider any two network diagrams equivalent when they are the same up to the equations of a cd-category and box re-labellings. Then the above yields a one-to-one correspondence between open DAGs and network diagrams \cite[Sections 3,5]{lorenz2023causal}. 

\begin{example} \label{ex:DAG}
Consider the open DAG $G$ over $\{X_1,X_2,X_3,X_4\}$ below, with output vertices $O=\{X_2, X_3\}$ circled, and with no input vertices. The equivalent network diagram $D_{G}$ is shown to the right. Note that the labels of the boxes are arbitrary. 
\[
\input{Graph-obs.tikz}
\qquad 
\iff 
\qquad 
\input{state-GO.tikz}
\]
\end{example} 

\begin{example} \label{ex:ODAG}
The following depicts an open DAG over $V = \{X_1,\dots,X_5\}$ with outputs $O=\{X_3, X_5\}$ and with inputs $I = \{X_2,X_3\}$ highlighted with special incoming arrows. To the right we show the corresponding network diagram with the same inputs and outputs.  
\[
\input{open-DAG.tikz} 
\qquad 
\iff 
\qquad 
\input{open-DAG-sd.tikz}
\]
\end{example}


We may now define generative models themselves, which involve specifying actual channels corresponding to the boxes in the network diagram. 

An \emph{interpretation} $\sem{-}$ of a network diagram $D$ in a cd-category $\catC$ consists of specifying an object $\sem{X_i}$ for each wire $X_i$ and channel $\sem{f} \colon \sem{X_1} \otimes \dots \otimes \sem{X_k} \to \sem{X}$ for each box $f$ in $D$ with inputs $X_1,\dots, X_k$ and output $X$.

\begin{definition}  \cite{lorenz2023causal}
Let $\catC$ be a cd-category. An \emph{open generative model} in $\catC$ is given by a network diagram $D$ along with an interpretation $\sem{-}$ in $\catC$. We call the objects corresponding to output wires \emph{observed} and the rest \emph{hidden}. We call such a model \emph{closed} when it has no inputs. 
\end{definition}

Note that an object of an open model may be both an input and output. In practice, we omit the $\sem{-}$ symbols and for each wire $X$ in the network diagram of a model denote the corresponding object $\sem{X}$ in $\catC$ also by $X$. Similarly for each box $c$ in the diagram with output $X$ we also write $c$ for the corresponding channel $\sem{c}$.

\begin{remark} 
Formally, an open generative model in our sense is the same as an \emph{open causal model} in the sense of \cite{lorenz2023causal}; that is, both have the same mathematical definition. However a `generative model' typically refers to a causal model with the extra interpretation of being possessed by a cognitive agent. 

Indeed, though not often stressed in the literature, a typical generative model in PP may be seen as a \emph{causal} Bayesian network, i.e. a \emph{causal model} in the sense of Pearl \cite{pearl2009causality}. This means that the probability channels which constitute the network do not represent arbitrary relationships but in fact (beliefs about) causal ones, such as how observations are caused by (rather than merely correlated with) hidden states of the world. For more discussion see Section \ref{sec:futurework}.
\end{remark}

Given any open generative model $\modelM$ we obtain an overall channel from its inputs to its outputs by composing the channels of the model, i.e. viewing the (interpreted) network diagram as a single channel in $\catC$. Often it is useful to also consider the following related channel. 


\begin{definition} 
Let $\modelM=(D,\sem{-})$ be an open generative model in $\catC$ with inputs $\Input$and outputs $O$. Let $S$ denote the non-input hidden (non-output) objects of the model. The \emph{total channel} $M$ of the model is the channel from $\Input$ to $S, O$:
\begin{equation} \label{eq:total-channel}
\input{total-channel2.tikz}
\end{equation}
given by interpreting the network diagram $D'$ in which we modify $D$ by adding an extra copy morphism to each object in $S$, to make it an output. 
\end{definition} 

Conversely, the usual channel from inputs to outputs is then simply the marginal over $S$: 
\begin{equation} \label{eq:total-vs-overall}
\input{totalvsoverall2.tikz}
\end{equation}
In particular for a closed generative model, with no inputs, we call the total channel the \emph{total distribution} of the model. It is a joint distribution over the hidden objects $S$ and observed objects $O$: 
\begin{equation} \label{eq:total-state} 
\input{total-state2.tikz} 
\end{equation}
with the original distribution over the observed objects as its marginal. 
\begin{equation} \label{eq:overall-state}
\input{overall-channel3.tikz}
\end{equation}

Let us now consider generative models in our main example category. 

\begin{example} \label{ex:CBN}
A closed generative model $\modelM$ in $\MatR$ is precisely a \emph{Causal Bayesian Network (CBN)}. This consists of specifying: 
\begin{itemize}
\item 
a finite DAG $G$with a subset $O \subseteq V$ of `observed' vertices and the remaining $S = V \setminus O$ being `hidden'; 
\item 
for each vertex $X_i$ an associated variable with a finite set of values also denoted $X_i$, and a mechanism $c_i$ given by a probability channel with density: 
\begin{equation} \label{eq:mechanism-density}
P(X_i \mid \Pa(X_i))
\end{equation}
The term `causal' refers to the fact each such mechanism has a causal interpretation. 
\end{itemize} 
Indeed, as we have seen, such a DAG $G$ with outputs $O$ is equivalent to a network diagram $D$ with no inputs. Specifying an interpretation of $D$ is then the same as choosing the sets $X_i$ of values and channels \eqref{eq:mechanism-density} for each box in the diagram. A CBN defines a joint distribution\footnote{
Often a Bayesian network is instead defined as a distribution $P(V)$ satisfying the \emph{Markov condition} \eqref{eq:markov} in terms of its conditionals. Since these conditionals may not be unique, and the channels $c_i$ are an important component of the model, we instead include the latter explicitly; for more discussion see \cite{lorenz2023causal}. 
} over all the variables $V=\{X_1,\dots,X_n\}$ with density 
\begin{equation} \label{eq:markov}
P(V) := \prod^n_{i=1} P(X_i \mid \Pa(X_i))
\end{equation}
which is precisely \eqref{eq:total-state}. The output distribution of the CBN is given by the marginal $P(O)$ over only the observed variables, corresponding to \eqref{eq:overall-state}. 
\end{example}

\begin{example}
An open generative model $\modelM$ in $\MatR$ is an `open CBN', where now for the input variables no channel \eqref{eq:mechanism-density} is specified. This induces via \eqref{eq:total-channel} the total channel
\[ 
P(S, O \mid \Incbn)
\]
from the inputs to the non-input hidden variables $S$ and output variables $O$, which here we would denote (the entries of) by $M(s, o \mid \inp)$. Its marginal $P(O \mid \Incbn)$ on the observed variables $O$ yields the channel $M(o \mid \inp)$ from \eqref{eq:total-vs-overall}. 
\end{example}

In short, a (closed) generative model in $\MatR$ specifies the internal structure of an output distribution $P(O)$ in terms of further variables and channels \eqref{eq:mechanism-density}, while an open generative model similarly specifies the internal structure of a channel $P(O \mid I)$ from inputs $I$ to outputs $O$. 


For the remainder of this section we will describe some of the common forms of (open) generative models which appear in PP. 

\subsection{Simple generative models} \label{sec:simplegenmodel}
By a \emph{generative model $S \to O$} we mean a generative model $\modelM$ with network diagram: 
\[
\input{GM-1.tikz}
\]
Thus $\modelM$ consists of objects $S, O$ with $O$ observed and $S$ hidden, a channel $c \colon S \to O$, called the \emph{likelihood}, and a distribution $\sigma$ on $S$, called the \emph{prior}. As alluded to earlier, we call $S$ the \emph{hidden states} and $O$ the \emph{observations} of the model. The total distribution of the model is given by 
\begin{equation} \label{eq:GM-total-state}
\input{total-state.tikz}
\end{equation}

More generally, we can consider an \emph{open} variant of such a generative model which now comes with a hidden object $\Input$ of \emph{inputs}, with the following network diagram:
\begin{equation} \label{eq:simple-open-model}
\input{simple-open-model.tikz}
\end{equation} 
Hence both the prior and likelihood now take an additional $\Input$ input. The total channel is given by 
\begin{equation} \label{eq:OGM-total-channel}
\input{total-channel-of-model.tikz}
\end{equation}
Intuitively, such an open model $\modelM$ consists of specifying a particular generative model $S \to O$ for each input in $\Input$.

\begin{example}
A generative model $S \to O$ in $\MatR$ consists of a finite set $S$ of hidden states, $O$ of observations, a likelihood channel $c(o \mid s)$ and prior distribution $\sigma(s)$. Often would often write $c(o \mid s)$ as simply $P(o \mid s)$ and $\sigma(s)$ as $P(s)$. 
We interpret $c(o \mid s)$ as the probability of observing $o$ when in the hidden state $s$. Then the total state \eqref{eq:GM-total-state} is the joint distribution over $S, O$ given by 
\[
M(s,o) = c(o \mid s) \sigma(s)
\]
and typically simply denoted $P(s,o)$. As the notation suggests $P(o \mid s)$ is the conditional and $P(s)$ the marginal of the joint distribution $P(s,o)$.

When the generative model is open as in \eqref{eq:simple-open-model} it now comes with a finite set $\Input$ of input values, with likelihood  $c(o \mid \inp, s)$ and prior $\sigma(s \mid \inp)$. The total channel \eqref{eq:OGM-total-channel} is then given by 
\[
M(s, o \mid \inp) = c(o \mid \inp,s)\sigma(s \mid \inp)
\]
Thus for each input $\inp$ we obtain a generative model $\modelM(\inp)$ of the form $S \to O$ and an induced joint distribution $M(\inp)$ over $S, O$.
\end{example}

\subsection{Discrete time models}

For a given $n \in \mathbb{N}$, a \emph{discrete time generative model} is a closed generative model $\modelM$ of the form 
\[
\input{GM3-simpler.tikz}
\]
Thus it consists of observed objects $O_1,\dots, O_n$, hidden objects 
$S_1, \dots, S_n$, a \emph{prior} distribution $D$ over $S_1$, \emph{observation} channels $\{A\colon S_t \to O_{t}\}^n_{t=1}$ and \emph{transition} channels $\{B\colon S_{t} \to S_{t+1}\}^{n-1}_{t=1}$. 
Typically we mean that there are fixed objects $S, O$ such that $S_j = S, O_j = O$ for all $j$, and similarly as our notation suggests all the $A$ and $B$ channels for each time step are taken to be identical.

We interpret the model $\modelM$ as describing the evolution of a system over discrete time steps from $t=1,\dots,n$. The system begins in its initial state with prior distribution $D$ and then evolves over each time step according to the transition channels $B$. Independently, we observe the system at each time $t$ via the channel $A$ to produce an observation in $O_t$.  

\begin{example}
A discrete time generative model in $\MatR$ is also called a  \emph{Hidden Markov Model}  or \emph{partially observable Markov decision process (POMDP)} \cite{parr2022active}.
\end{example}

\subsection{Policy models} \label{sec:disc-time-policies}

We now introduce an explicit ingredient whereby the agent can model its own \emph{actions}. As in reinforcement learning  \cite{tschantz2020reinforcement}, a choice of actions or behaviour is called a \emph{policy}. In a discrete time setting, a policy can be thought of as determining likely sequences of actions over the time steps, which in turn influence the evolution of the states over time. 

An \emph{$n$-time step model with policies} is a generative model $\modelM$ of the form:
\[
\input{GEN-models-4.tikz}
\]
Thus it now includes a hidden object $P$ of \emph{policies} which forms an input to each transition channel $B$ from $S_t, P$ to $S_{t+1}$, for $t \leq n-1$. The model also comes with a prior distribution $\habits$ over $P$, which are called the \emph{habits} of the system. Note that here the policy the system is undertaking is considered hidden.

Again we typically take $S_j = S$ and $O_j  = O$ for some fixed objects $S, O$, with all channels $A$ identical and all $B$ channels identical. 

\begin{example}
Models of this form, within $\MatR$, are the central examples used in the active inference tutorial \cite{smith2022step} and book \cite{parr2022active}.
\end{example}

\subsection{Hierarchical models} 

Central to much of PP is the study of \emph{hierarchical} generative models \cite{de2017factor,parr2022active}, which have a natural graphical description. These are generative models given by composing various open generative models in layers, where the outputs of the open models in one layer match the inputs of the models in the next layer, such as in the example below. 
\begin{equation} \label{eq:hierarchical-model}
\input{hierarchy-again.tikz}
\end{equation}
Here it is understood that each box $M_j$ represents an open generative model, which we may decompose further in terms of its own network diagram with inputs and outputs as shown. The right-hand labels indicate that the input wire  to $M_1$ has type $S^{(0)}$, the output wires from $M_1$ both have type $S^{(1)}$ etc \footnote{It is also common to introduce a labelling convention for the wires such as $S^{(0,1,2)}$ where the indices represent wire numbers in each layer as we read up the diagram. However this quickly becomes unwieldy, and in most cases the graphical description of the network is the most convenient.   
}. 

We interpret the inputs to each (box within a) layer as a `control' signal from the layer below. Note that because we read diagrams bottom to top, the layers further down the diagram are in fact those usually referred to as more `high-level' or `higher' in the hierarchy. 

The structure of the model tells us that the `high-level' features cause the generation of the `lower-level' features. For example $S^{(0)}$ could describe an overall action policy while the $S^{(3)}$ control more fine-grained motor actions. Another common example explored in \cite{de2017factor} is where the output wires from each box denote individual time steps. In this case time runs faster in the lower-level layers (higher in the diagram). For example in the diagram above six time steps occur in layer $S^{(3)}$ for every time step in layer $S^{(1)}$.   

Plugging in the network diagrams for each open model corresponding to $M_j$ yields a composite network diagram for the whole hierarchical model. For example in the following hierarchical model, the network diagrams for $M_1$ and  $M_2$ are shown in the highlighted boxes below and compose to yield the diagram on the right-hand side. 
\[
\input{simple-hierarchy2.tikz}
\]
Much of the PP literature concerns such hierarchical models and the passing of these `top-down predictions' (the flow of information up the diagram in this case) are adjusted by `bottom-up errors' passed back down the model. The latter takes place when a model is \emph{updated}, which we address next.

\section{Updating Models} \label{sec:updating} 

Consider an agent with be a simple generative model $\modelM$ of the form $S \to O$ as in Section \ref{sec:simplegenmodel}. Recall that this induces a joint distribution $\Mso$ over $S, O$ as in \eqref{eq:total-state}, whose marginal on $S$ is the prior $\sigma$ describing `beliefs' about how likely each state in $S$ is to occur.
\[
\input{prior-marginal.tikz}
\]
Now suppose the agent receives an observation, which in general may be `soft', given by an distribution $\obs$ over $O$. The agent would like to \emph{update} these beliefs to obtain a new \emph{posterior} distribution over $S$, describing how likely each $s \in S$ now is given the observation.
\[
\input{Mso.tikz}
\  \ , \  \ 
\input{obs.tikz}
 \ \quad 
\mapsto 
 \ \quad 
\input{posterior.tikz}
\]
How then should the agent update the marginal on $S$ to yield this posterior? 
For a general soft observation with distribution $\obs$ over $O$ there are at least two distinct but natural ways to carry out Bayesian-style updating, as pointed out by Jacobs in \cite{jacobs2019mathematics}, which we describe in this section. When the observation $\obs$ is sharp, however, corresponding to a single element $o \in O$, there is a canonical way to carry out this belief updating, usually simply referred to as \emph{Bayesian} updating, which we introduce first. 

\subsection{Sharp Updating} 

Let us begin by describing updates with respect to a sharp observation, given by (a point distribution at) an element $o \in O$. Such Bayesian updating is closely related to the notion of \emph{conditional} probabilities, which have a nice characterisation in cd-categories. Here we follow the approach to conditioning from \cite{lorenz2023causal}, building on earlier treatments \cite{coecke2012picturing,cho2019disintegration,fritz2020synthetic}; see also \cite{di2023evidential}.

\begin{definition} 
Let $\catC$ be a cd-category, and $\bijointstate$ a joint distribution over $X, Y$. Then  \emph{a conditional} of $\bijointstate$ by $Y$ is a morphism $\bijointstate|_Y \colon Y \to X$ such that the following holds: 
\begin{equation} \label{eq:is-a-conditional}
\input{Bayesianomega.tikz}
\end{equation}
where $\sigma$ is the marginal of $\bijointstate$ on $Y$. 
If $\catC$ has normalisation and cancellative caps, we define \emph{the (minimal) conditional} to be the morphism:
\[
\input{mincondomega.tikz}
\] 
\end{definition}

Each minimal conditional is indeed a conditional as shown in the Appendix of \cite{lorenz2023causal}. As we saw for normalisations, a conditional is only a partial channel in general, being a channel only when $\bijointstate$ has `full support'.

\begin{example}
In $\MatR$ the minimal conditional $\bijointstate|_Y$ is given by 
\[
\bijointstate|_Y(x \mid y) := \Norm{x} \bijointstate(x,y) = \frac{\bijointstate(x,y)}{\sum_{x'} \bijointstate(x',y)}
\]
whenever the sum in the denominator is non-zero, and $\bijointstate|_Y(x \mid y) = 0$ for all $x$ otherwise. Thus when $\bijointstate$ is normalised with density denoted $P(X,Y)$ this is the usual conditional $P(X \mid Y)$. The condition \eqref{eq:is-a-conditional} amounts to the usual `chain rule' $P(x,y) = P(x \mid y)P(y)$ for the probability distribution $P(x,y) = \omega(x,y)$, since $\sigma(y)$ is the marginal $P(y)$ and we have:
\[
\omega(x,y) 
\ \ = \ \ 
\input{chainrule.tikz}
\ \ = \ \  
\bijointstate|_Y(x \mid y) \sigma(y)
\]
\end{example}

For a generative model $\modelM$ of the form $S \to O$ with joint distribution $M$ over $S, O$ we call the minimal conditional $M|_O \colon O \to S$ the \emph{Bayesian inverse} of the model. It specifies how to update beliefs about $S$ for each specific sharp observation $o \in O$. Explicitly, given a sharp distribution $\obs=\delta_o$ over $O$ for some $o \in O$ the updated beliefs are given by the posterior:
\begin{equation}
\label{eq:sharp-update-one}
\input{sharpupdate.tikz}
\end{equation}

\begin{example}
In $\MatR$, for a sharp observation $\delta_o$ for some $o \in O$, the posterior is the distribution over $S$ given by the usual Bayesian update:
\begin{equation} \label{eq:sharp-updating-fla}
M(s \mid o) = \frac{M(s, o)}{\sum_{s'} M(s',o)}
\end{equation}
\end{example}



\subsection{\pearl and \jeffrey Updating}

There are two distinct ways to generalise updating to the case of a soft observation given by a distribution $\obs$ over $O$, described in \cite{jacobs2019mathematics}. Diagrammatically these correspond to generalising from either the former or latter diagrams in  \eqref{eq:sharp-update-one}. For more on both forms of updating in cd-categories see also the treatment by Di Lavore and Román  \cite{di2023evidential}. 

\begin{definition} \label{Def:jeffpearl}
Let $\catC$ be a cd-category with normalisation and cancellative caps, and $M$ a joint distribution over $S, O$. Given a distribution $\obs$ over $O$, the \emph{\jeffrey update} denoted $\jup{M}{\obs}$ or $M|_{\obs}$ is given by the composite $M|_O \circ \obs$, i.e.: 
\[
\input{jeff-state.tikz}
\]
whenever this is normalised, and more generally is given by the normalisation of the above state. The \emph{\pearl update} denoted $\pup{M}{\obs}$ or $M|^{\obs}$ is instead given by the normalisation: 
\[
\input{pearl-update.tikz}
\]
recalling that the effect $\obs$ is given by composing $\obs$ with a cap as in \eqref{eq:eff-from-state}.
\end{definition}

\begin{example}
For a generative model $\modelM$ from $S$ to $O$ in $\MatR$, with joint distribution $M$ over $S, O$, the \jeffrey update is given by 
\begin{align}
\jup{M}{\obs}(s) \label{eq:Jeffrey-update-explicit}
= \ExpVal_{o \sim \obs} \Norm{s}  M(s,o) 
=
\sum_{o } \frac{M(s, o) \obs(o)}{\sum_{s'} M(s',o)} 
\end{align}
while the \pearl update is 
\begin{align}
\pup{M}{\obs}(s)  \label{eq:Pearl-update-explicit}
 = \Norm{s}  \ExpVal_{o \sim \obs} M(s,o)
=
\frac{\sum_{o } M(s, o) \obs(o)}
{\sum_{s',o'} M(s',o') \obs(o')} 
\end{align}
\end{example}

The distinction between both update procedures is not always considered in the literature since for sharp observations they coincide with the usual Bayesian update. Indeed the following is immediate from \eqref{eq:sharp-update-one}.



\begin{lemma}
Let $\catC$ be a cd-category with normalisation and cancellative caps. Then for each sharp state $o$ on $O$ the updates coincide: $\jup{M}{o} = \pup{M}{o} = M|_O \circ o$. 
\end{lemma}



In contrast, for a general observation $\obs$ the two updates differ in the way they apply normalisation, amounting to whether one normalises with respect to (or separately from) the observation itself. 
\[
\input{pearl-isnot-jeff.tikz}
\]
The \jeffrey update simply composes the observation $\obs$ with the Bayesian inverse (partial) channel $M|_O$. If $M|_O$ is only a partial channel the result may not be normalised (such as when $\obs$ falls outside the support), in which case the update is then further normalised. The \pearl update instead involves a single normalisation, taking place after composing with the observation, so that  $\obs$ is inside the normalisation box. 


\begin{remark}
Jacobs has compared the two forms of updating within $\MatR$ in detail, noting that their inferences can differ considerably, but that both can be considered rational notions of updating \cite{jacobs2019mathematics}. One difference between the updates is that by definition \jeffrey updating forms a probability channel in $O$ (whenever $M|_O$ is a channel, i.e. $M$ has full support over $O$). 
In contrast, the normalisation over $\obs$ in \pearl updating means that it does not form a channel in $O$. 
The two update procedures can also be characterised by the following respective properties. For a generative model $\modelM$ over $S, O$ with likelihood $c$, \jeffrey updating  minimises the KL-divergence between $\obs$ and the marginal on $O$ of the updated model in which we replace the prior with the posterior (left-hand below).  \pearl updating instead has the property that it maximizes the expected value of the function $\obs$ (right-hand below).
\[
\jup{M}{\obs} \text{   minimises:   } \quad 
D_{KL} \left( \input{KL.tikz} \  \  , \   
 \input{obs.tikz} \right) 
\qquad \qquad  \qquad \qquad
\pup{M}{\obs}\text{   maximises:   } 
\quad
\input{pearl-update-max.tikz}
\]
The PP literature has mostly focused on updating with respect to sharp observations, in which the two notions coincide. It is an interesting question for the future to determine which (if either) form of updating is most natural in Bayesian models of cognition. 
\end{remark}

\subsection{Updating Open Models}

Since a typical generative model in PP is composed of various \emph{open} generative models $\modelM$, it is also important to describe how an agent may update such open models $\modelM$, now coming with inputs $\Minput$. In this case we consider the induced channel 
$\pMso \colon \Minput \to S, O$. The prior beliefs about $S$ are now given by the marginal $\pMs \colon \Minput \to S$, which we can think of specifying beliefs over $S$ for each input $\inp \in \Minput$. Given an observation $\obs$ over $O$ the agent now wishes to update this to a posterior channel of the same kind. 
\[
\input{prior-marginal-channel.tikz}
\ \qquad 
\mapsto 
 \ \qquad 
\input{posterior-channel.tikz}
\]

All of the treatment of updating above generalises straightforwardly to such open models, amounting to updating the corresponding closed model $M(\inp)$ over $S, O$ for each input $\inp \in \Minput$. 

Explicitly, for any morphism $f \colon X \to Y \otimes Z$ in a cd-category, \emph{a conditional} is any morphism $f|_Z$ satisfying the left-hand equation below, where $\sigma$ is the corresponding marginal of $f$. In the presence of normalisation and cancellative caps, \emph{the (minimal) conditional} is that given on the right below, as in \cite{lorenz2023causal}.  
\[
\input{disint-channelf.tikz}
\qquad \qquad \qquad 
\input{disint-channel-picf.tikz}
\]

\begin{definition}
Let $M \colon \Minput \to S, O$ be the channel induced by an open model $\modelM$, and $\obs$ a distribution over $O$, in a cd-category with normalisation and cancellative caps.  The \emph{\jeffrey update} denoted $\jup{M}{\obs}$ or $M|_{\obs}$ is given by composing $M|_O$ with $\obs$ as left-hand below (or more generally by its normalisation if the result is not a partial channel).  The \emph{\pearl update} denoted $\pup{M}{\obs}$ or $M|^{\obs}$ is instead given as on the right-hand side.
\[
\input{jeff-channel2.tikz} \qquad \qquad \qquad \input{pearl-channel-box.tikz}
\]
\end{definition}
 By the defining property of normalisations \eqref{eq:norm-sup-cond} the \pearl update $M|_\obs$ satisfies the following, which will be useful later. 
\begin{equation} \label{eq:Pearl-rule}
\input{pearl-channel-2.tikz}
\end{equation}

\begin{example}
In $\MatR$ the minimal conditional of $f$ by $Z$ is given by 
\[
f|_Z(y \mid x ,z) = \Norm{y} f(y, z \mid x)
\]
and for a probability channel $P(Y, Z \mid X)$ corresponds to the usual conditional $P(Y \mid X, Z)$. The formulae for both updates $\jup{M}{\obs} (s \mid i), \pup{M}{\obs}(s\mid i)$ are the same as \eqref{eq:Jeffrey-update-explicit}, \eqref{eq:Pearl-update-explicit} simply replacing each $M(s,o)$ term with $M(s,o \mid \inp)$, i.e. 
\begin{align*}
\jup{M}{\obs}(s \mid i) &= \ExpVal_{o \sim \obs} \Norm{s}  M(s,o \mid i ) 
\\
\pup{M}{\obs}(s \mid i) &= \Norm{s}  \ExpVal_{o \sim \obs} M(s,o \mid i )
\end{align*}
Again both update procedures coincide with $M(s \mid o, \inp)$ for sharp observations $\inp \in \Minput$ and all inputs $\inp \in \Minput$. 


\end{example}

\begin{remark}
Di Lavore and Román also study both forms of updating in cd-categories in which (non-chosen) conditionals exist in \cite{di2023evidential}, calling them `partial Markov categories'. There updating is defined via arbitrary (non-minimal) conditionals, meaning that $M|_O$ can be arbitrarily defined outside the support on $O$ of $M$. However since this arbitrary choice can impact the result of a \jeffrey update $M|_O \circ \obs$ when $\obs$ is also non-zero outside this support, we instead define updating via the minimal conditional $M|_O$.
\end{remark}

\section{Perception and Planning} \label{sec:percep-planning}

Let us now see how the notion of updating is applied by an agent to govern its behaviour in PP. Two fundamental uses of updating are the following. 

\para{Perception}
Firstly, as already alluded to, we can consider the case of an agent with a generative model $\modelM$ from $S$ to $O$, interpreted as 
accounting for observations $O$ in terms of hidden states of the world $S$. For example, $O$ may be the space of pixel-level descriptions of images while $S$ is a compressed representational space of possible objects which the images portray. 

Given an observation encoded by a (soft or sharp) distribution $\obs$ over $O$, the agent can update its prior over hidden states $S$ to obtain a posterior describing how likely each hidden state is to have caused the observation. We refer to this general process of updating as \emph{perception} and view the resulting distribution as the agent's specific perception of the  observation $\obs$. Intuitively perception takes the `raw data' of the observation $\obs$ and returns (a distribution over) representations $S$. 
\[
\input{perception-posterior.tikz}
\]
Intuitively, the update answers the question \emph{`Given that I have received this observation, how likely is each possible world state?'}. In the literature this is often referred to as \emph{inference}, in reference to Bayesian inference.

\para{Planning}
A second application of updating by an agent is in \emph{planning} its behaviours. Here an agent possesses a generative model $\modelM$ of the same formal structure but with objects labelled $P, F$ and interpreted differently. Now $P$ encodes the action policies, or behaviours, the agent may carry out, while $F$ represents observations (or states) it may receive in the \emph{future}. The model $\modelM$ includes a prior over policies which we can think of as the agent's habits or typical behaviours. 

Here the agent possesses  some \emph{preferences} about which future observations (or states) are most desirable, encoded by a distribution $\prefs$ over $F$. Intuitively, the distribution will have highest density on the most desirable outcomes. The agent can then plan its actions by updating its habits with respect to these preferences:
\[
\input{planning-posterior2.tikz}
\]
The process of deriving this distribution can intuitively be called `planning'. We can think of this update as answering the question \emph{`Given that I will obtain my preferences in the future, how likely is each policy to have led to this outcome?'}.


The resulting `plan' distribution over $P$ can be used to guide the agent's future behaviour. For example, an agent may then sample an policy to pursue from this distribution, so that the more probable policies according to the distribution are more likely to be carried out. 



\section{Exact Active Inference} \label{sec:exact-act-inf}

Both uses of updating by an agent, planning and perception, come together in the concept of \emph{active inference}, of which we are now able to present a fully formal diagrammatic account. 

 Consider an agent possessing a generative model describing how its actions, in the form of action policies $P$, bring about changes in its observations. These consist of both observations for the present time (and previous times) $\Opres$ and for future time steps $\Ofuture$. Thus the agent has a closed generative model $\modelM$ of the following form.  
\begin{equation} \label{eq:closed-model}
\input{model1.tikz}
\end{equation}
Here (abusing notation slightly) we denote by $M$ also the channel from policies to observations induced by the model, and $\habits$ is the prior over policies describing the agent's habits.

Suppose further that the agent's model explains the observations at each of these time steps through hidden states, where $\Spres$ denotes the hidden states in the present time and $\Sfuture$ in the future, so that we have:
\begin{equation} \label{eq:ActinfopenM}
\input{modelOF2.tikz}
\end{equation}
for `observation' channels $A, A'$ and `transition' channels $B, B'$. The induced distribution on $P, \Opres, \Ofuture$ is then given by: 
\begin{equation} \label{eq:actinf-total}
\input{actinf-total.tikz}
\end{equation}

The goal of active inference is then the following. The agent receives a current observation given by a distribution $\obs$ over $\Opres$, and also carries a distribution $\prefs$ describing its preferences for future observations $\Ofuture$. 
The agent then wishes to update its prior $\habits$ over policies to yield a posterior which describes its plan of action\footnote{Ultimately, having derived their `plan' distribution the agent may then sample a single action policy $\pi \in P$ as in Section \ref{sec:percep-planning}, and act accordingly. We imagine that via the true `generative process' in the world (distinct from the agent's model) this leads to further observations in the future, to which the agent carries out further planning steps, and so on. Our focus is simply on a single step of how the agent derives their `plan' from $\obs$ and $\prefs$.}:
\begin{equation} \label{eq:plan-act-inf}
\input{OF.tikz} \qquad \mapsto \qquad \input{actinf-posterior.tikz}
\end{equation}
Intuitively, the posterior over policies can be thought of as answering the question \emph{`Given that I have received this observation $\obs$ now, and will attain my preferences $\prefs$ in the future, which action policy am I pursuing?'}. Note that, perhaps surprisingly, the agent's own action policy is thus treated as hidden from itself, and something that it must infer.

Now, typically the objects above all decompose into further structure, as in the following example.

\begin{example}
A common application of active inference is to the discrete-time models with policies given in Section \ref{sec:disc-time-policies}, which we may view as instances of \eqref{eq:ActinfopenM} as follows. Consider such a model featuring $N$ time-steps, where $n << N$ is considered the current time, and all times $m$ with $n \leq m \leq N$ as in the future. The spaces of `current' hidden states and observations $\Spres, \Opres$ are the products over all previous time-steps $t=1,\dots,n$ up to and including the current time, while the future hidden states and observations $\Sfuture, \Ofuture$ take the product over all future time-steps $t=n+1,\dots,N$. 
\begin{align*}
\Spres := S_1 \otimes \dots \otimes S_n \qquad \qquad 
\Sfuture := S_{n+1} \otimes \dots \otimes S_N \\ 
\Opres := O_1 \otimes \dots \otimes O_n  \qquad \qquad 
\Ofuture := O_{n +1} \otimes \dots \otimes O_N
\end{align*}
The observation channels in the overall model \eqref{eq:ActinfopenM} would then be given by:
\[
\input{obsovertimes.tikz}
\]
while the transition channels are as follows:
\[
\input{DisctimeB.tikz}
\]
so that the composite \eqref{eq:actinf-total} yields the network diagram for the overall model for times $t=1,\dots,N$.
\end{example}



An agent may employ various update procedures, such as those discussed in Section \ref{sec:updating}, to calculate its plan of action \eqref{eq:plan-act-inf}.  Though both forms of updating coincide for sharp inputs, and the observations $\obs$ in the active inference literature are typically taken to be sharp, the preferences $\prefs$ are often not; that is, there may be multiple desirable future observations in $\Ofuture$. Thus \pearl and \jeffrey updating can be expected to differ.

Here we will describe an exact active inference procedure based on \pearl updating, allowing both observations $\obs$ and preferences $\prefs$ to be soft. We leave the exploration of \jeffrey updating in active inference for future work. 

Now let us consider how the agent can in the ideal case compute its plan \eqref{eq:plan-act-inf} via an exact update procedure. Firstly, let us rewrite the channel in \eqref{eq:ActinfopenM} as follows. 
\[
\input{open-model-shape-new.tikz}
\]
Here the channels $\Mpres$, $\Mfuture$ are the compositions indicated by the highlighted boxes\footnote{While we could define $\Mfuture$ without $\Sfuture$ as an output, the appearance of  $\Sfuture$ will be useful later in treating approximate active inference. Note also that the dashed boxes in this case do not denote normalisation.}. Now applying the property of \pearl updates \eqref{eq:Pearl-rule} to $\Mpres$ we have the following:
\begin{equation} \label{eq:pearl-exact}
\input{pearl-exact2.tikz}
\end{equation}
Here we have again denoted by $\Mpres, \Mfuture$ their respective marginals on $\Opres, \Ofuture$, given by discarding $\Spres, \Sfuture$ respectively. In the last step we used associativity of copying and the following argument:
\[
\input{condprop-proof.tikz}
\]
where in the middle step we used \eqref{eq:norm-axioms} and \eqref{eq:norm-axioms-more}  to slide the channel $M_2$ and copying out of the normalisation box, respectively. 

Thus we obtain an exact expression for active inference.

\begin{proposition} \label{prop:exact-act-inf}
The plan over policies in Pearl-style exact active inference is given by:
\begin{equation} \label{eq:act-inf-exact-pic}
\input{act-inf-exact.tikz}
\end{equation}
In $\MatR$ the plan has density over policies $\pi \in P$ given by: 
\begin{align} \label{eq:exact-act-inf-formula}
\text{\normalfont plan}(\pi) &:= \Norm{\pi} \left( E(\pi)
(\obs \circ \Mpres(\pi))
(\prefs \circ \ucond{M}{\obs}(\pi)) \right)
\\ 
&= \Norm{\pi} \input{exact-ai-pic.tikz}
\end{align}
\end{proposition}

\begin{proof}
The first equality holds by definition, so $\text{\normalfont plan}(\pi) = \Norm{\pi} f(\pi)$ where $f$ is the density of the state in \eqref{eq:pearl-exact}. But this is given by: 
\[
\input{pearl-exact-plugin.tikz}
\]
using that $\pi$ is sharp, where the three right-hand scalars are precisely the terms in \eqref{eq:exact-act-inf-formula}. The last line comes from noting that the given marginal $M_1 \colon P \to A$ is precisely $B \circ A$.
\end{proof}

There is only one problem with this form of active inference: the quantity \eqref{eq:exact-act-inf-formula} is completely intractable to calculate. Along with the normalisation in calculating $\ucond{M}{\obs}$, calculating the terms in \eqref{eq:exact-act-inf-formula} would involve summation (or integration) over $\Spres,\Opres$ and $\Sfuture,\Ofuture$ respectively, requiring us to respectively calculate:
\[
\sum_{s \in \Spres, o \in \Opres} \obs(o) A(o \mid s) B(s \mid \pi) 
\qquad \qquad 
\sum_{o' \in \Ofuture} \prefs(o') \ucond{M}{\obs}(o' \mid \pi) 
\]

To make the calculation of these updates tractable, an agent in active inference is understood to instead use a special form of approximation scheme, to which we now turn.  



\section{Free Energy} \label{sec:free-energy}

We have seen that for an agent to perform exact Bayesian updating is computationally intractable. In active inference, an agent instead carries out approximate updating by minimising a quantity known as \emph{free energy} \cite{friston2006free,friston2010free,parr2022active}. In this section for simplicity we work concretely in the category $\catC = \MatR$, though the same notions should be similarly defined in continuous settings.

The extra mathematical ingredient\footnote{
To study free energy we will move beyond a purely diagrammatic approach and make use of some probabilistic calculations, most notably to define `surprise'. However later in Section 9 we will see how to represent surprise in diagrams (via `log-boxes'). In future work it would be interesting to represent all of the calculations in this section using such diagrams. 
} needed to define free energy will be the following .

\begin{definition} 
For any distribution $\sigma$ over $X$ and $x \in X$ we define the \emph{surprise} as $\Sp(\sigma)(x) := -\log \sigma(x)$. For another distribution $\omega$ on $X$ we define the overall \emph{surprise} of $\sigma$ relative to $\omega$ as the expectation value:
\[
\Sp\left(\input{omegastatealt.tikz} \ , \ \input{sigmastatealt.tikz} \right) \quad := 
\quad \mathop - \mathop {\mathbb{E}}_{x \sim \omega}  \log \sigma(x)
\]
The \emph{entropy} $\He(\omega)$ of $\omega$ is its self-surprise: 
\[
\He \left( \input{omegastatealt.tikz} \right) := \Sp\left(\input{omegastatealt.tikz} \ , \ \input{omegastatealt.tikz}  \right)
\]
while the \emph{Kullback-Liebler (KL) divergence} $\DKL(\omega,\sigma)$ from $\sigma$ to $\omega$ is the difference between these quantities:
\[
\DKL \left(\input{omegastatealt.tikz} \ , \ \input{sigmastatealt.tikz} \right) := \Sp\left(\input{omegastatealt.tikz} \ , \ \input{sigmastatealt.tikz}  \right) - \He \left( \input{omegastatealt.tikz} \right) 
\]
\end{definition} 

The KL divergence is a commonly used similarity measure on distributions, with $D(\omega, \sigma) \geq 0$ and $D(\omega, \omega) = 0$ for all distributions $\omega, \sigma$.

We may now define the following general notion of free energy. Throughout we consider a distribution $M$ over $S, O$, which we imagine to be induced by a generative model from $S$ to $O$. In this section for simplicity given any such distribution we denote its marginals on $S, O$ and conditional channels $M|_S, M|_O$ again simply by $M$.

\begin{definition}[Free Energy]
The \emph{Free Energy} of a distribution $\qFE$ over $S, O$ relative to $M$ is defined as:
\begin{align} \label{eq:FE-functional}
\FE 
\left( \input{Q-gens.tikz} \ , \input{M-gens.tikz}\right)
 \ \ \ &:= \ \ \ 
\Sp \left(\input{Q-gens.tikz} \ , \input{M-gens.tikz}\right)
 \ - \  
\He \left ( \ \input{QS-2.tikz} 
\right) 
\end{align}
\end{definition} 
Explicitly then we can re-write the free energy in the following useful form.
\begin{align}
\FE(Q,M) &= 
\mathop{\mathbb{E}}_{(s,o) \sim Q}[\log(\qFE(s))-\log(M(s,o))] \\ 
&=  \label{eq:FE-otos}
\mathop{\mathbb{E}}_{(s,o) \sim Q}[\log(\qFE(s)-\log(M(s \mid o)) - \log M(o)] \\ 
&= 
\mathop{\mathbb{E}}_{o\sim Q} \Sp \left( \input{Q-o.tikz} \ , \input{M-o.tikz} \right) 
 + \Sp\left(\input{QO3.tikz},\input{MO3.tikz} \right) 
-
\He \left(\input{QS-2.tikz} \right)
\label{eq:FE-graphical}
\end{align}

We now turn to two specific variants of this quantity commonly considered in active inference.  

\subsection{Variational Free Energy}

Suppose an agent receives an observation given by a distribution $\obs$ over $O$, and wishes to perform an approximate Bayesian update of its prior beliefs about $S$ as encoded by the marginal of $M$ on $S$. It may do so by finding the distribution $q$ over $S$ which minimises the following quantity.

\begin{definition}[Variational Free Energy]
Given a distribution $M$ over $S, O$ and distribution $\obs$ over $O$, the \emph{Variational Free Energy (VFE)} of a distribution $q$ over $S$ is defined as: 
\begin{align*}
\VFE\left( \input{q.tikz} \right)
&:= 
\FE \left(\input{q.tikz} \ \input{o.tikz}  \ ,  \ \input{M-gens.tikz} \right)
\end{align*}
\end{definition} 

An important feature of the VFE is the following. Using the expression \eqref{eq:FE-graphical} and pulling the entropy term inside the expectation we see that  
\begin{align}
\VFE(q)
 \ \ &= \ \  
\mathop{\mathbb{E}}_{o\sim \obs} 
\DKL
\left( \input{q.tikz} \ ,  \  \input{M-od.tikz} \right) 
   \ + \  \Sp\left(\input{o-state-small.tikz},\input{MO3.tikz} \right) 
  \label{eq:VFE-exact}
\\ 
&\geq \ \  
\DKL
\left( \input{q.tikz} \ , \   \input{M-o2.tikz} \right) 
 \ + \  \Sp\left(\input{o-state-small.tikz},\input{MO3.tikz} \right) 
\label{eq:VFE-approx}
\end{align} 
The inequality follows from concavity of the KL divergence and Jensen's inequality, which states that for any probability measure $\omega$ on $X$, measurable function $f \colon X \to \mathbb{R}$ and concave function $\phi$ on $\mathbb{R}$ we have 
\begin{equation} \label{eq:jensen-general}
\ExpVal_{x \sim \omega}[\phi(f(x))] \leq \phi(\ExpVal_{x \sim \omega} [f(x)])
\end{equation}

In particular we see that the inequality \eqref{eq:VFE-approx} will be a strict equality whenever $\obs=\delta_o$ is given by a sharp observation $o \in O$. In this case the minimum VFE value is given by the exact Bayesian inverse $M|_o$, with value $\VFE = -\log M(o)$. Hence for a sharp observation $o$, minimising the VFE minimises the KL-divergence between $q$ and the Bayesian inverse $M|_o$, achieving approximate inversion $q \approx M|_o$. Moreover $\VFE(q)$ is an upper bound on the surprise of the observation $o$, and when $q \approx M|_o$ we have $\VFE(q) \approx S(o,M)$.

\para{VFE Updating}
This process of minimising VFE to compute an approximate Bayesian update is central in active inference, but typically only considered for such sharp observations. Here we can now consider the more general minimisation of VFE for a soft observation given by a distribution $\obs$. In fact we may view this as another notion of updating for a prior over $S$, in addition to the two forms of updating met in Section \ref{sec:updating}. 

Firstly, observe that in the expression \eqref{eq:VFE-exact} since the surprise term is constant, the distribution $q$ which minimises $\VFE(q)$ will be that which minimises the left-hand expected KL term, which is equal to the following.
\begin{align*}
\mathop{\mathbb{E}}_{s\sim q} \left[
\log q(s) - 
\mathop{\mathbb{E}}_{o\sim \obs} 
\log M(s \mid o)
\right]
\end{align*}
This quantity will in turn be minimised when this expression over $S$ is equal to a constant $K$, so that:
\begin{align*}
\log q(s) = \mathop{\mathbb{E}}_{o\sim \obs} \left[
\log M(s \mid o)\right] + K
\end{align*} 
The distribution $q$ will be given by normalising $q(s)$ in the above expression, allowing us to ignore the constant and yielding the following notion of updating motivated by the VFE. Recall that the \emph{softmax} of a function $f \colon X \to \mathbb{R}^+$ is defined by $\softmax(f)(x) = \Norm{x} {e^{f(x)}}$.

\begin{definition}[VFE Update]
Given a joint distribution $M$ over $S, O$ and distribution $\obs$ over $O$ the \emph{VFE update} is the posterior 
\begin{align} \label{eq:VFE-update}
\vup{M}{\obs}(s)  &= 
\Norm{s}  e^{\ExpVal_{o \sim \obs} \log M(s \mid o)}
\\ 
&= 
\softmax( 
\mathop{\mathbb{E}}_{o\sim \obs} 
\log M(s \mid o) )
\end{align}
where $\softmax$ denotes a softmax over $S$. 

Similarly, for any channel $M$ from $P$ to $S, O$ we define the \emph{VFE update} of its marginal $P \to S$ point-wise, by $\vup{M}{\obs}(s \mid \pi) = \vup{M(\pi)}{\obs}(s)$ for each $\pi \in P$. 
\end{definition}

From the derivation above we see that $q= \vup{M}{\obs}$ is the distribution which minimises $\VFE(q)$. Note that, as for our other forms of updating, for a sharp observation $\obs = \delta_o$ we have $\vup{M}{\obs}(s) = M(s \mid o)$.

To relate general VFE minimisation for a soft observation to expectation values, we will use the following form of approximation. Firstly, note that by Jensen's inequality, for any probability measure $\omega$ and real function $f$ over $X$ we have: 
\begin{equation} \label{eq:logapprox}
e^{\ExpVal_{x \sim \omega}[\log f(x)]} \leq \ExpVal_{x \sim \omega} [f(x)]
\end{equation} 
Whenever we take both sides of such an inequality to be approximately equal, let us say we are using a \emph{log approximation}. In particular for any distributions $\omega, \sigma$ on $X$ the follow holds log-approximately:
\begin{equation} \label{eq:log-approx}
e^{-\Sp}\left(\input{omegastatealt.tikz} \ , \ \input{sigmastatealt.tikz} \ \right)  \quad \lessapprox \quad 
\input{validity.tikz}
\end{equation}
Indeed this states precisely \eqref{eq:logapprox} for the case 
$f(x) = \sigma(x)$. Such approximations can be used to relate free energy to exact expectation values, as follows.

\begin{proposition} \label{prop:VFE-approx}
Let $\vup{M}{\obs}$ be the VFE update of $M$ relative to a distribution $\obs$ over $O$, and $\VFE$ its VFE value. Then the following holds log-approximately: 
\[
\input{VFEapprox.tikz}
\]
\end{proposition}
\begin{proof}
Define $f(s) := e^{\mathop{\mathbb{E}}_{o \sim \obs}  \log M(s \mid o)}$ and the normalisation constant $K = \sum_s f(s)$, so that $K \vup{M}{\obs}(s) = f(s)$. Then we have:
\begin{align}
F &=
\Sp(\obs, M) + 
\mathop{\mathbb{E}}_{s \sim q} [\log \vup{M}{\obs}(s) - \mathop{\mathbb{E}}_{o \sim \obs}  \log M(s \mid o)]  \label{eq:VFE-min-formula}
 \\ &=  \Sp(\obs, M) - \log K \nonumber
\\ \nonumber
e^{-F} \vup{M}{\obs}(s) &= 
e^{-\Sp(\obs, M)} K \vup{M}{\obs}(s) \\ \nonumber
& = e^{-\mathop{\mathbb{E}}_{o \sim \obs} [\log M(o) + \log M(s \mid o)]} \\ \nonumber
& = e^{-\mathop{\mathbb{E}}_{o \sim \obs} [\log M(s, o)]} 
 \approx \mathop{\mathbb{E}}_{o \sim \obs} M(s, o)
\end{align}
where in the last step we used a log-approximation.
\end{proof}

\begin{remark}
 Compare the formula for VFE update to the \jeffrey and \pearl updates \eqref{eq:Jeffrey-update-explicit}, \eqref{eq:Pearl-update-explicit}. While the \jeffrey update composes the conditional $O \to S$ with $\obs$ exactly, the VFE update instead minimises the expected KL below. 
\[
\input{jeffrey-def2.tikz}
\qquad 
\text{ while }
\qquad 
\vup{M}{\obs} \ \ \text{ minimises } \ \ 
\ExpVal_{o\sim \obs}  \ 
\DKL 
\left( \input{vfeup.tikz} \ , \  \input{M-od.tikz} \right) 
\]
\end{remark}

\subsection{Expected Free Energy}

A second form of free energy employed in active inference is used by an agent with a model featuring a space $\EFO$ describing observations in the future. It then has a distribution $\prefs$ over $\EFO$ modelling preferences for these future observations. Rather than updating its beliefs about future states, the agent simply want to assess how well the marginal of the model on $\EFO$ will fit these preferences, via the following approximation. 

\begin{definition} 
Given a distribution $M$ over $\EFS, \EFO$ and distribution $\prefs$ over $\EFO$, the  \emph{Expected Free Energy (EFE)} is defined as 
\begin{equation}\label{eq:EFE-def}
\EFE \left( \input{M-genso.tikz} \ , \input{Cstate2f.tikz} \right)
 \ \ := \ \  
\FE \left( \input{M-genso.tikz} \ \ , \ \ \input{model-CF.tikz}  \right) 
\end{equation}
\end{definition}

The EFE compares the given model $M$ to the right hand generative model which perfectly attains the preferences, via its marginal $\prefs$ over $\EFO$, whilst making use of the same inverse channel $\EFO \to \EFS$. Writing the EFE explicitly, and then rewriting in terms of the typically more readily computable channel $\EFS \to \EFO$, we have 
\begin{align*}
\EFE (M,\prefs)
 &= 
\mathop{\mathbb{E}}_{(s,o) \sim M}[\log(M(s))-\log(M(s \mid o))] - \mathop{\mathbb{E}}_{o \sim M}[\log \prefs(o)] \\ 
&=
\mathop{\mathbb{E}}_{\substack{s \sim M \\ o \sim M \circ s} }
[-\log(M(o \mid s)] 
 + \mathop{\mathbb{E}}_{o \sim M}[ \log M(o) - \log \prefs(o)]
\\
&= \label{eq:entropy-expression}
\mathop{\mathbb{E}}_{s \sim M}
\left[
H
\left(
\input{Ms.tikz}
\right)
\right]
+
D
\left(
\input{MO.tikz} \ ,
\input{C5.tikz}
\right)
\end{align*}

The final line expresses the EFE in terms of a right-hand \emph{risk} term, which assesses how well the predicted state over $\EFO$ matches the preferences $\prefs$, and a left-hand \emph{uncertainty} term given by the expected entropy in the observations. Thus minimising EFE requires both matching preferences and reducing uncertainty. For more interpretations of EFE see \cite{parr2022active}.  

Now using Jensen's inequality and the concavity of entropy, one may show that for any distribution $\omega$ and channel $c $ we always have: 
\[
\mathop{\mathbb{E}}_{x \sim \omega} 
H
\left( \input{cx.tikz} \right)
\leq 
H \left( \ 
\input{comega.tikz} \ 
\right)
\]
Hence the EFE is bounded above by the surprise of the preferences:
\begin{align} \label{eq:EFE-surprise-bound}
\EFE \left( \input{M-genso.tikz} \ , \input{Cstate2f.tikz} \right) \ \ 
\leq
\ 
H
\left(
\input{MO.tikz}
\right)
\ 
+
\ 
D
\left(
\input{MO.tikz} \ ,
\input{C5.tikz}
\right)
 \  
=
\ 
\Sp
\left(
\input{MO.tikz} \ ,
\input{C5.tikz}
\right) 
\end{align}
Thus minimising the EFE results in reducing the surprise of the preferences, making them more likely to be obtained according to the model. Taking the inequality to be an approximation and applying exponentials to both side along with a log-approximation then gives the following. 

\begin{proposition} \label{prop:EFE-approx}
The EFE is bounded above and approximately equal to the expectation value:
\[
e^{- \EFE} \left( \input{M-genso.tikz} \ , \input{Cstate2f.tikz} \right)   \quad \lessapprox 
\quad \input{validity3alt.tikz}
\] 
\end{proposition}

\subsection{Free Energy in Active Inference}

We conclude this section by noting two uses of free energy in approximate active inference, treated in the next section. For these we now consider a channel $M$ from $P$ to $S, O$, typically induced by an open model. For each $\pi \in P$ this specifies a joint distribution $M(\pi)$ over $S, O$, to which we may apply free energy calculations. 

\begin{corollary} \label{cor:open-VFE-EFE}
Let $\obs$ and $\prefs$ be distributions over $O$. Let $\vup{M}{\obs} \colon P \to S$ be the VFE update of $M$ by $\obs$, and for each $\pi \in P$ set $\VFE(\pi) := \VFE(\vup{M(\pi)}{\obs})$ to the corresponding VFE value. Similarly for each $\pi \in P$ let $\EFE(\pi) = \EFE(M(\pi),\prefs)$. Then we have the following approximations: 
\[
\input{approx-simpleb.tikz} \qquad \qquad \qquad \input{efe-approx2.tikz}
\]
\end{corollary}

In the above the effect $e^{-\VFE}$ is given by $\pi \mapsto e^{-\VFE(\pi)}$, for $\pi \in P$, and $e^{-\EFE}$ is defined similarly.

\begin{proof}
For the first approximation, plugging in a (sharp state given by) an element $\pi \in P$ to both sides shows that this is equivalent to Proposition \ref{prop:VFE-approx} holding for each joint distribution $M(\pi)$ over $S, O$ with respect to the observation $\obs$. For the second approximation, apply Proposition \ref{prop:EFE-approx} to the joint distribution $M(\pi)$ over $\EFS, \EFO$ for each $\pi \in P$.
\end{proof}

\section{Active Inference via Free Energy} \label{sec:approx-actinf}

Let us now return to the situation of an agent carrying out active inference as in Section \ref{sec:exact-act-inf}. As before the agent's generative model $\modelM$ in \eqref{eq:closed-model} consists of its habits $\habits$ over policies $P$ and a channel $M$ from $P$ to current and future observations $\Opres, \Ofuture$, factoring via current and future hidden states $\Spres, \Sfuture$. Given its observation $\obs$ and future preferences $\prefs$ it can now use free energy to give a viable approximation of its updated plan of behaviour from Proposition \ref{prop:exact-act-inf}, proceeding in two steps.  We saw already in \eqref{eq:pearl-exact} that: 
\[
\input{pearl-exact3.tikz}
\]
\para{`Perception' step} 
In the first step, the agent approximately updates the part of the model pertaining to the current time, $\Mpres$, in light of the observation $\obs$. For each policy $\pi$ it computes a distribution $q(\pi)$ with (approximately) minimal VFE $\VFE(q(\pi))$, thus obtaining a channel $q \colon P \to S$ which approximates the VFE update of $\Mpres$ by $\obs$. For each $\pi \in P$ denote the corresponding VFE value by $\VFE(\pi)$. Explicitly: 
\begin{align*}
\VFE(\pi)
=
\VFE \left( \input{qpii.tikz} \right)
&:= 
\FE \left(\input{qpii.tikz} \ \input{o.tikz}  \ ,  \ \input{Mprespi.tikz} \right)
\end{align*}

\para{`Prediction' step}
In the second step, the agent uses this approximation channel $q$, to obtain a channel $\Mq$ which approximates the model over future states and observations, defined as follows: 
\[
\input{Mq.tikz}
\]
For each policy $\pi$ this induces a distribution $\Mq(\pi)$ over $\Sfuture, \Ofuture$, for which the agent can compute the EFE with respect to the preferences:
\begin{equation} \label{eq:EFE-pi}
\EFE(\pi) := \EFE \left( \input{Mqpi.tikz}  \  ,  \  \input{prefF.tikz}  \right)
\end{equation}

Using these free energy quantities, the agent may carry out approximate active inference. The following formula is central in the active inference literature.


\begin{theorem} \label{thm:approx-ai}
The agent can carry out approximate active inference given observation $\obs$ and preferences $\prefs$ by setting its plan to have density 
\begin{equation} \label{eq:active-inference-functional}
\text{\normalfont plan}(\pi) := \softmax(\log E(\pi) - F(\pi) - G(\pi))
\end{equation}
where $\softmax$ denotes a softmax over $\pi \in P$.
\end{theorem} 

\begin{proof}
We have:
\[
\input{actinf-approx-1.tikz}
\]
where we used Corollary \ref{cor:open-VFE-EFE} in both approximation steps. Thus defining our plan as on the left-hand below yields an approximate update:
\[
\input{plan-norm.tikz}
\]
But finally, note that the left-hand distribution is precisely given by \eqref{eq:active-inference-functional}. Indeed for each $\pi \in P$, corresponding to a sharp effect on $P$, we have: 
\[
\input{softmax-derivation.tikz} \quad = \quad \habits(\pi) \  e^{-\VFE(\pi)} \ e^{-\EFE(\pi)}
\]
Hence the normalisation of the above is precisely the softmax expression \eqref{eq:active-inference-functional}.
\end{proof}

This formula for active inference via free energy, though frequently used, is usually only justified in a fairly heuristic manner \cite{parr2022active}. Previous accounts rely on the less clear notion of treating EFE as a `prior' to updating\footnote{Despite the fact EFE is not straightforwardly a component of the generative model, and requires inference over present states $S$ to be calculated first, rather than prior to them}. Here we have instead seen how the expression can be derived from a direct diagrammatic argument, directly from the structure of the generative model.

\section{Compositionality of Free Energy} \label{sec:comp-FE}

A crucial aspect of active inference is the idea that an agent can be understood to minimise free energy at all levels, so that it may be seen to globally minimise free energy in its generative model by minimising free energy within each component.

To formalise this idea we must first introduce a notion of free energy for open models. For this we will make use of the following graphical notation for the surprise. 
\begin{definition}
Given any effect $e$ on $X$ in $\MatR$, corresponding to a function $e \colon X \to \mathbb{R}^+$, we denote by 
\[
\input{logbox.tikz}
\]
the function $-\log e (x) \colon X \to (-\infty, \infty]$.
\end{definition}

\begin{remark}
Note that a log-box is no longer an effect within $\MatR$, since when $e(x) = 0$ we will have $-\log e(x) = \infty$. Here we will interpret any diagram involving log-boxes with inputs $X_1, \dots, X_n$ and no outputs as a (formula specifying a) function $X_1 \times \dots \times X_n \to (-\infty, \infty]$. Composing boxes in the diagram amounts to summation over wires, as for $\MatR$. Given two such diagrams $D_1, D_2$ we write $D_1 + D_2$ for the function given by their point-wise sum as functions. In future it would be interesting to explore a formal categorical semantics for log-boxes. 
\end{remark} 

In particular we can apply a log box to any distribution $\omega$ in $\MatR$ by first turning it into an effect, yielding the surprise $S(\omega)(x) = -\log \omega(x)$. 
\[
\input{statetologbox.tikz}
\]
Similarly for a pair of distributions $\omega, \sigma$ we have 
\begin{equation} \label{eq:surprise-logbox}
\Sp\left(\input{omegastatealt.tikz} \ , \ \input{sigmastatealt.tikz} \right) \quad = \quad \input{logsurprise.tikz} .
\end{equation}

From the properties of the logarithm, one may verify that log-boxes then satisfy the following compositional properties. 

\begin{lemma} \label{lem:log-axioms}
For all effects $d, e$ and sharp states $x$ the following hold.
\begin{enumerate} 
\item \label{enum:lbox1}
\[
\input{logbox1.tikz}
\]
\item \label{enum:lbox2}
\[
\input{logbox2.tikz}
\]
\item \label{enum:lbox3}
\[
\input{logbox3.tikz}
\]
\item \label{enum:lbox4}
\[
\input{logboxsharp.tikz}
\]
\end{enumerate}
\end{lemma}
\begin{proof}
Plugging inputs $x, y$ into each equation they reduce to the following respective properties of the logarithm. \eqref{enum:lbox1}: $\log(d(x)e(x)) = \log(d(x)) + \log(e(x))$. \eqref{enum:lbox2}: $\log(1) = 0$. \eqref{enum:lbox3}: $\log(d(x)e(y)) = \log d(x) + \log e(y)$. \eqref{enum:lbox4} holds by definition, since both diagrams are given by $y \mapsto \log d(x,y)$. 
\end{proof} 

The following properties then follow from diagrammatic reasoning, using the relation between caps and copying. 

\begin{proposition} \label{prop:logboxchanrule}
\ 
\begin{enumerate}
    \item \label{enum:surprisetensor}
    For all effects $d, e$ and normalised states $\sigma, \omega$:
\[ 
\input{surprisetensor.tikz}
\]
In particular, entropy is additive across parallel composition: $H(\sigma \otimes \omega) = H(\sigma) + H(\omega)$.

\item \label{enum:de-ext}
For all effects $d, e$:
\[
\input{logboxwider.tikz}
\]
\item \label{enum:Logboxchannelrule}
For all morphisms $f, g$:
\[
\input{logboxchannelrule2.tikz}
\]
\end{enumerate}
\end{proposition}
\begin{proof}
\eqref{enum:surprisetensor} follows from Lemma \ref{lem:log-axioms} \eqref{enum:lbox3} since:
\[
\input{stensproof.tikz}
\]
\eqref{enum:de-ext} follows from Lemma \ref{lem:log-axioms} \eqref{enum:lbox1}, \ref{enum:lbox2} and \ref{enum:lbox3} since: 
\[
\input{lbwideproof.tikz}
\]
\eqref{enum:Logboxchannelrule} is a special case of \eqref{enum:de-ext} where we define the effects $d, e$ by composing $f, g$ with caps on their output, respectively, since using the relation between caps and copying we see that:
\[
\input{logboxchanruleproof.tikz}
\]
\end{proof}


Now, by construction, for any joint distributions $M, Q$ over $S, O$ the free energy is given by:
\[
\FE 
\left(
\input{Q-gens.tikz} \ , \input{M-gens.tikz}\right)
\quad 
=
\quad
\input{logboxQ.tikz}
\]
Hence for any joint distribution $M$ over $S, O$ and distributions $q, \obs$ over $S, O$ respectively, the VFE is given by:
\begin{align*}
\VFE\left( \input{q.tikz} \right)
\quad
= 
\quad
\input{logboxvfe.tikz}
\end{align*}

We can use this to define a generalisation of VFE for (the channels induced by) open generative models.

\begin{definition}[Open VFE]
Given a channel $M \colon \Inpcfe \to S, O$, distribution $q$ over $\Inpcfe, S$ and distribution $\obs$ over $O$, we define the \emph{open Variational Free Energy} as 
\begin{equation} \label{eq:open-VFE}
\VFE\left(
\input{MISO.tikz} \ , \ 
\input{qps2.tikz} \ , \ \input{obs3.tikz} \right)
\quad
:= 
\quad
\input{openVFE.tikz}
\end{equation}
\end{definition}

In the special case where $\Inpcfe$ is trivial, the open VFE coincides with the usual VFE.


We can use the compositional properties of log boxes to show that this form of free energy is compositional, in an appropriate sense. First consider the following two ways in which we may compose open models. 

Consider a pair of open models $\modelM_1, \modelM_2$ with inputs $\Inpcfe_1, \Inpcfe_2$, outputs $O_1, O_2$ and hidden states $S_1, S_2$, respectively, such that $O_1 = \Inpcfe_2$. We can compose these in sequence into a single open model $\modelM$ from $\Inpcfe_1$ to $O_2$, with $S_1, S_2, O_1$ as its hidden states, with induced total channel $M$ from $\Inpcfe_1$ to the remaining variables given as below. 
\begin{equation} \label{eq:sequential-openmodel}
\input{modelcomp1.tikz}
\qquad 
\text{ with total channel }
\qquad 
\input{Mcompchan.tikz}
\end{equation}
Formally, the left-hand diagram is a composition in the category of open causal models; see \cite[Sec. 5]{lorenz2023causal}.

We can also compose open models in parallel. Given two open models $\modelM_i$ with inputs $\Inpcfe_i$, outputs $O_i$ and hidden states $S_i$, for $i=1,2$ we can define an open model $\modelM$ with both sets of inputs, outputs and hidden states with induced induced total channel $M$ from $\Inpcfe_1, \Inpcfe_2$ to the remaining variables given as below. 
\begin{equation} \label{eq:parallel-openmodel}
\input{modelcomp2.tikz}
\qquad 
\text{ with total channel }
\qquad 
\input{Mcompchan2.tikz}
\end{equation}
For each of these forms of composition of open models, we wish to establish that free energy is compositional in that the VFE of the open model $\modelM$ is determined from the VFE of its constituents. This ensures that locally minimising VFE (within each of sub-component) can achieve global VFE minimisation also.




\begin{theorem} \label{thm:compFE1}
For the sequential composite model $\modelM$ in \eqref{eq:sequential-openmodel} with $O_1=\Inpcfe_2$, with total channel $M$, and any distribution $\obs$ over $O_2$ and distributions $q_1, q_2$, the following holds:
\begin{equation} \label{eq:comp-openFE1}
\VFE \left( \input{Mcompchana.tikz}  , 
 \input{qtens.tikz} \
,  \input{obso2.tikz} \right)
\ \ = 
\ \ 
\VFE \left( \input{M1.tikz} \ , \input{q1.tikz} \ , \input{obso1.tikz} \right)
\ \ + \ \ 
 \VFE \left( \input{M2.tikz} \ , \input{q2.tikz} \ , \input{obso2.tikz} \right)
\end{equation}
where 
\begin{equation} \label{eq:obs1induced}
\input{obso1fla.tikz}
\end{equation}
\end{theorem}

Intuitively, \eqref{eq:obs1induced} expresses the way in which the beliefs about inputs $\Inpcfe_2$ in $M_2$ are passed down to the model $\modelM_1$ as observations in $O_1=\Inpcfe_2$. 

\begin{proof}
After rearranging some wires, we have that
\begin{align*}
\VFE (M,q,\obs)  \ \ \ 
 &= \input{comp-proof2.tikz}
\end{align*}
where for the first two terms we apply Proposition \ref{prop:logboxchanrule} with $f=M_1$ and $g=M_2$, along with the fact that $\obs$ and $q_1$ are normalised, and the second two terms are from Proposition \ref{prop:logboxchanrule} \eqref{enum:surprisetensor}. But this is precisely the right-hand side of \eqref{eq:comp-openFE1}.
\end{proof}

Next let us turn to the parallel composite of open models. 

\begin{theorem} \label{thm:compFE2}
For any channels and distributions $M_i, q_i, \obs_i$ for $i=1,2$, the following holds.
\begin{equation} \label{eq:comp-openFE2}
\VFE \left( 
\input{M1M2.tikz} \
 ,  \input{qtens.tikz} \ ,  \input{obstens.tikz} \right)
\ \ = 
\ \ 
\VFE \left( \input{M1.tikz} \  , \input{q1.tikz} \ , \input{obso1.tikz} \right)
\ \ + \ \ 
 \VFE \left( \input{M2.tikz} \ , \input{q2.tikz} \ , \input{o2.tikz} \right)
\end{equation}
\end{theorem}
\begin{proof}
The left-hand side is given by 
\[
\VFE (M,q,\obs)  \ \ \ 
= \input{openVFEtensor.tikz}
\]
which is precisely the right-hand side, where we applied Proposition \ref{prop:logboxchanrule} \eqref{enum:surprisetensor}.
\end{proof}

The above results tell us that an agent with an overall generative model may minimise VFE by minimising VFE locally within each sub-model, an important property underlying the application of the free energy to all levels of a system.

\section{Outlook} \label{sec:futurework}

In this article we have aimed to give a concise formulation of active inference in terms of string diagrams interpreted in a cd-category $\catC$, focusing on the case of finite discrete systems as described by $\catC=\MatR$. In particular we were able to derive the formula for approximate active inference via free energy minimisation purely from the high-level structure of a generative model undertaking active inference, and derived a compositionality property for free energy. 

However these are just the first steps towards a fully compositional account of intelligent behaviour according to predictive processing, and there are many directions for future work.

\para{Message passing}
So far we only studied active inference at a high-level, saying that an agent must, for each observation, arrive at an updated distribution $q$ by free energy minimisation, without discussing how this is to be carried out. In PP this minimisation is normally achieved via so-called `message passing' algorithms, such as `variational' and `marginal' message passing \cite{parr2019neuronal}. These are defined on undirected graphical models described by Forney factor graphs, induced by a generative model. In future it would be interesting to include a categorical account of message passing within our framework, to complete our description of active inference. 

\para{Continuous settings}
Another technical matter would be to extend the treatment of PP beyond the finite case to further cd-categories describing continuous settings, such as a suitable category of Gaussian probabilistic processes, which are widely employed in PP under the `Laplace assumption'. One issue is in extending our treatment of minimal conditionals to such continuous settings, where they are not as straightforwardly defined.



\para{Causal reasoning}
We have here pointed out that an generative model may be seen precisely as a causal model  \cite{pearl2009causality}. In future it would be interesting to explore how an agent may carry out causal reasoning on its model using concepts from the causal model framework such as `interventions', as treated graphically in \cite{lorenz2023causal}, and how such reasoning relates to active inference.

\para{Approximations}
The treatment of active inference via free energy in Section \ref{sec:approx-actinf} relied on applying various approximation steps from Section \ref{sec:free-energy} to parts of the diagram. Certainly more could be done to set bounds on how well these approximations hold, including how they extend from part of a diagram to the whole generative model. 

\para{Updating within PP}
The categorical perspective led us to naturally consider soft observations (given by distributions) rather than the usual sharp ones (given by points), which come with distinct notions of Jeffrey and \pearl updating (Section \ref{sec:updating}), as well our new notion of VFE updating (Section \ref{sec:free-energy}). While we were able to describe active inference via the latter two forms of updating, it would be interesting to compare against \jeffrey updating and establish which form of exact updating is most naturally considered (and approximated) in PP. That is, given that both forms of updating have different goals \cite{jacobs2019mathematics}, which one (if either) is approximately carried out by the brain? This question was also raised in \cite{di2023evidential}. 

 We note that \pearl updating can be more generally defined with respect to any effect (see e.g. \cite{di2023evidential}), i.e. any (not necessarily normalised) function. There is disagreement between active inference and reinforcement learning (RL) in whether an agent's preferences should, rather than as a distribution as in active inference, be simply modelled by a function $\prefs \colon F \to \mathbb{R}^+$ assigning a `value' in $\mathbb{R}^+$ to each possible future observation, i.e. as an effect $\prefs$ on $F$ \cite{friston2009reinforcement,tschantz2020reinforcement}. In this case \pearl updating may be the most natural to treat planning. In contrast, \jeffrey updating may be most fitting for perception, with an observation $\obs$ naturally encoded as a distribution i.e. a `fuzzy point' in $O$.

 


\para{Consciousness in PP}
Various proposals have been put forward for how PP and active inference can be related to consciousness. Continuing from previous work from two of the authors on IIT \cite{kleiner2021mathematical,TullKleiner}, in future we hope to account for these proposals within our graphical account of active inference.

% ...

\para{Categorical modifications of PP}
Beyond simply recasting previous results in PP categorically, in future one may also study what new insights the compositional perspective may bring to PP and active inference, and to connect the work to ongoing research within categorical cybernetics \cite{smithe2021cyber,capucci2021towards} and more broadly to the research programme of compositional intelligence.

\bibliographystyle{alpha}
\bibliography{pp.bib}

\end{document}

%% file: preamble.tex
\newcommand{\papertitle}{Active Inference in String Diagrams: A Categorical Account of Predictive Processing and Free Energy}

\newcommand{\para}[1]{\paragraph{#1}}

\newcommand{\pearl}{Pearl~} 
\newcommand{\jeffrey}{Jeffrey~} 
\newcommand{\pup}[2]{#1^{}_P} 
\newcommand{\jup}[2]{#1^{}_J} 

\newcommand{\vup}[2]{#1_F} 
\newcommand{\ucond}[2]{#1|^{#2}} 

\newcommand{\iD}{\raisebox{0.125em}{\tiny $\blacktriangleright$}}
\newcommand{\cbb}[1]{   
\color{RoyalPurple}
\begin{itemize}[label=\iD,topsep=.5em,itemsep=.5em,leftmargin=*]
\item \textbf{#1}
    \begin{itemize}[label=-,itemsep=.2em,topsep=.2em,leftmargin=*]
}
\newcommand{\cbe}{        
    \end{itemize}
\end{itemize}
\color{black}
}

\newcommand{\totchan}[1]{#1} 
\newcommand{\prior}[1]{\sigma} 
\newcommand{\upd}[2]{\mathrm{update}({#1},{#2})} 

\newcommand{\Mq}{M_q}

\newcommand{\Mso}{\totchan{M}} 
\newcommand{\Ms}{\prior{M}}
\newcommand{\Mo}{\sigma'} 

\newcommand{\oMs}{M|_O}
\newcommand{\Mspost}{\upd{M}{\obs}}
\newcommand{\Mspostsharp}{\upd{M}{o}}

\newcommand{\pMso}{\totchan{M}}
\newcommand{\pMs}{\Ms} 
\newcommand{\pMspost}{\Mspost} 
\newcommand{\poMs}{\oMs}
\newcommand{\pMo}{\Mo}

\newcommand{\Minput}{\Input}
\newcommand{\Input}{I} 
\newcommand{\inp}{i} 
\newcommand{\Inpcfe}{I} 
\newcommand{\Incbn}{I} 

\newcommand{\monunit}{\mathsf{I}} 

\newcommand{\pMof}{M}
\newcommand{\Mpof}{M}

\newcommand{\qFE}{Q}
\newcommand{\EFS}{S}
\newcommand{\EFO}{O}

\usepackage[dvipsnames]{xcolor}
\usepackage[utf8]{inputenc}
\usepackage{graphicx}
\usepackage[left=2.5cm,right=2.5cm,top=2.5cm,bottom=2.5cm]{geometry}
\usepackage{caption}
\usepackage{subcaption}
\usepackage{lscape}
\usepackage{amsmath}
\usepackage{float}
\allowdisplaybreaks		
\usepackage{amsfonts}
\usepackage{amssymb}
\usepackage{braket}
\usepackage{longtable}
\usepackage{dsfont}
\usepackage{hyperref}
\usepackage{xurl}
\usepackage{enumitem}
\usepackage{stmaryrd}
\usepackage[utf8]{inputenc}
\usepackage{graphicx}
\usepackage[left=2.5cm,right=2.5cm,top=2.5cm,bottom=2.5cm]{geometry}
\usepackage{xcolor}
\usepackage{lscape}
\usepackage{amsthm}
\usepackage{tikz}  
\usepackage{tikzit}
\usetikzlibrary{shapes,decorations,arrows,calc,arrows.meta,fit,positioning}
\usetikzlibrary{arrows.meta}

\theoremstyle{plain}
\newtheorem{Def}{Definition}
\newtheorem{definition}[Def]{Definition}
\newtheorem*{definition*}{Definition}
\newtheorem{theorem}[Def]{Theorem}

\newtheorem{proposition}[Def]{Proposition}
\newtheorem{lemma}[Def]{Lemma}

\newtheorem{remark}[Def]{Remark}
\newtheorem{corollary}[Def]{Corollary}

\newtheorem{example}[Def]{Example}

\usepackage{amsmath,amssymb,stmaryrd,xspace}
\usepackage{amsthm}
\usepackage{tikz-cd}
\usepackage{multirow}
\usepackage{mathtools}
\usepackage{url}
\usepackage{relsize}
\usepackage{bm} 

\usepackage{hyperref}
\usepackage{mathrsfs}

\newcommand{\lapprox}{\approx_l}

\newcommand{\Opres}{O}
\newcommand{\Ofuture}{F}

\newcommand{\Spres}{S}
\newcommand{\Sfuture}{S'}

\newcommand{\Mpres}{M_1}
\newcommand{\Mpreso}{M_1}
\newcommand{\Mfuture}{M_2}
\newcommand{\Mfuturef}{M_2}

\DeclareMathOperator{\DKL}{D}
\DeclareMathOperator{\He}{H}
\DeclareMathOperator{\Sp}{S}
\DeclareMathOperator{\FE}{FE}
\DeclareMathOperator{\VFE}{F}
\DeclareMathOperator{\EFE}{G}

\newcommand{\obs}{\textbf{o}} 

\newcommand{\bijointstate}{\omega}

\newcommand{\prefs}{C}
\newcommand{\habits}{E}

\newcommand{\softmax}{\sigma}

\newcommand{\ExpVal}{\mathop{\mathbb{E}}} 

\DeclareMathOperator{\norm}{Norm}
\newcommand{\Norm}[1]{\mathop{\norm}_{#1}}

\newcommand{\Gauss}{\cat{Gauss}}

\newcommand{\Pa}{\mathsf{Pa}}

\usepackage{tikzit}

\newcommand{\iset}[1]{\mathbf{#1}} 

\newcommand{\mech}[1]{\mathbb{#1}}
\newcommand{\model}[1]{\mech{#1}}

\newcommand{\modelM}{\model{M}}

\newcommand{\channel}{\mathrm{channel}}

\definecolor{darkblue}{rgb}{0.2,0.2,0.7}
\definecolor{darkgreen}{rgb}{0.0, 0.5, 0.0}
\definecolor{ForestGreen}{RGB}{34,139,34}

\definecolor{planning}{rgb}{0.2,0.2,0.7}

\newcommand{\rl}[1]{\textcolor{black} {#1}}

\newcommand{\Kl}{\mathrm{Kl}}

\newcommand{\FStoch}{\cat{FStoch}}

\newcommand{\MatR}{\cat{Mat}_{\mathbb{R}^+}}

\usepackage{tikz-cd}

\newcommand{\sem}[1]{\ensuremath{\llbracket #1 \rrbracket}}

\DeclareMathOperator{\ext}{ext}

\usepackage{stackengine}

\newcommand{\cat}[1]{\ensuremath{\mathbf{#1}}}
\newcommand{\catC}{\cat{C}}

\newcommand{\Rel}{\cat{Rel}}

\newcommand{\id}[1]{\ensuremath{\mathrm{id}_{#1}}}

\newcommand{\hilbH}{\mathcal{H}} 

\usepackage{tikz,xypic}

\tikzstyle{every picture}=[baseline=-0.25em,scale=0.5]

\newenvironment{picc}[1][]
{\begin{aligned}\begin{tikzpicture}[font=\tiny,#1]}
{\end{tikzpicture}\end{aligned}}
\newenvironment{pic}[1][] {\begin{aligned}\begin{tikzpicture}[scale=2.0, font=\tiny,#1]}{\end{tikzpicture}\end{aligned}} 

\usetikzlibrary{circuits.ee.IEC}
\pgfdeclarelayer{edgelayer}
\pgfdeclarelayer{nodelayer}
\pgfsetlayers{background,edgelayer,nodelayer,main}

\tikzstyle{whitedott}=[circle, draw=black, fill=white, inner sep=.4ex]
\tikzstyle{greydott}=[circle, draw=black, fill=black!25, inner sep=.4ex] 
\tikzstyle{blackdott}=[circle, draw=black, fill=black, inner sep=.4ex]

\tikzstyle{upgroundsmall}=[circuit ee IEC, thick, ground, rotate=90, scale=1.5, tikzit fill=black, tikzit shape=rectangle]

\newcommand{\discard}[1]{\ensuremath{\tinygroundnew_{#1}}}

\newcommand{\tinygroundnew}{
\smash{
{\hspace{-3pt}
\ensuremath{
\begin{picc}[scale=1.0] 
    \node[upgroundsmall, xscale=0.8, yscale=0.7] (1) at (0,0.16) {};
    \draw (0,0.03) to (0,-0.25);
\end{picc}
}\hspace{-1pt}}}}

\newcommand{\tinymultflip}[1][whitedott]{
\smash{\raisebox{-2pt}{\hspace{-5pt}\ensuremath{\begin{pic}[scale=0.4,yscale=1]
    \node (0) at (0,0) {};
    \node[#1, inner sep=1.5pt] (1) at (0,0.55) {};
    \node (2) at (-0.5,1) {};
    \node (3) at (0.5,1) {};
    \draw (0.center) to (1.center);
    \draw (1.center) to [out=left, in=down, out looseness=1.5] (2.center);
    \draw (1.center) to [out=right, in=down, out looseness=1.5] (3.center);
    \node[#1, inner sep=1.5pt] (1) at (0,0.55) {};
\end{pic}
}\hspace{-3pt}}}}

\newcommand{\tinycomult}[1][whitedott]{
\smash{\raisebox{-2pt}{\hspace{-5pt}\ensuremath{\begin{pic}[scale=0.4,yscale=1]
    \node (0) at (0,0) {};
    \node[#1, inner sep=1.5pt] (1) at (0,0.55) {};
    \node (2) at (-0.5,1) {};
    \node (3) at (0.5,1) {};
    \draw (0.center) to (1.center);
    \draw (1.center) to [out=left, in=down, out looseness=1.5] (2.center);
    \draw (1.center) to [out=right, in=down, out looseness=1.5] (3.center);
    \node[#1, inner sep=1.5pt] (1) at (0,0.55) {};
\end{pic}
}\hspace{-3pt}}}}

\newcommand{\tinycopy}{\tinycomult[whitedott]}

\newcommand{\tinycap}{\smash{\raisebox{-3pt}{\hspace{-2pt}\ensuremath{\begin{pic}[scale=0.2, yscale=-1]
   \pgftransformscale{1.5} \draw[scale = 1] (0,0) to[out=-90,in=-90,looseness=1.5] (1.5,0);
\end{pic}}}}}

\input{tikzit.tikzstyles}

\tikzstyle{tikzfig}=[]

\tikzset{arrow/.style={decoration={
    markings,
    mark=at position #1 with \arrow{>[length=2pt, width=3pt]}},
    postaction=decorate},
    reverse arrow/.style={decoration={
    markings,
    mark=at position #1 with {{\arrow{<[length=2pt, width=3pt]}}}},
    postaction=decorate}
}

%% file: tikzit.tikzstyles
\tikzstyle{label}=[font={\footnotesize}, text height=1ex, text depth=0.15ex]
\tikzstyle{map}=[draw, shape=rectangle, inner sep=2pt, minimum height=5mm, fill=white, minimum width=5mm]
\tikzstyle{medium map}=[draw, shape=rectangle, inner sep=2pt, minimum height=5mm, fill=white, minimum width=12mm, tikzit fill=red]
\tikzstyle{large map}=[draw, shape=rectangle, inner sep=2pt, minimum height=5mm, fill=white, minimum width=18mm, tikzit fill=blue]
\tikzstyle{scalar}=[circle, draw, inner sep=2pt, line width=0.7pt]
\tikzstyle{upground}=[circuit ee IEC, thick, ground, rotate=90, scale=1.5, tikzit fill=black, tikzit shape=rectangle]
\tikzstyle{downground}=[circuit ee IEC, thick, ground, rotate=-90, scale=1.5, tikzit fill=black, tikzit shape=rectangle]
\tikzstyle{downgroundnorm}=[circuit ee IEC, thick, ground, rotate=-90, scale=1.5, fill=white, tikzit shape=rectangle, tikzit fill=red]
\tikzstyle{point}=[fill=white, draw, shape=isosceles triangle, shape border rotate=-90, isosceles triangle stretches=true, inner sep=0.2pt, minimum height=0.8mm, yshift=-0.0mm, tikzit shape=rectangle, minimum width=0.5cm]
\tikzstyle{copoint}=[fill=white, draw, shape=isosceles triangle, shape border rotate=90, isosceles triangle stretches=true, inner sep=0.2pt, minimum width=0.5cm, minimum height=0.8mm, yshift=-0.0mm, tikzit shape=rectangle, minimum width=0.5cm]
\tikzstyle{wide point}=[fill=white, draw, shape=isosceles triangle, shape border rotate=-90, isosceles triangle stretches=true, inner sep=0.2pt, minimum height=0.8mm, yshift=-0.0mm, minimum width=12mm, tikzit shape=rectangle, tikzit fill=red]
\tikzstyle{wide copoint}=[fill=white, draw, shape=isosceles triangle, shape border rotate=90, isosceles triangle stretches=true, inner sep=0.2pt, minimum width=0.5cm, minimum height=0.8mm, yshift=-0.0mm, tikzit shape=rectangle, minimum width=12mm, tikzit fill=red]
\tikzstyle{whitedot}=[circle, draw=black, inner sep=.4ex, fill=white]
\tikzstyle{greydot}=[circle, draw=black, inner sep=.4ex, fill=grey]
\tikzstyle{blackdot}=[circle, draw=black, inner sep=.4ex, fill=black]
\tikzstyle{decomp}=[fill=white, draw, shape=isosceles triangle, shape border rotate=-90, isosceles triangle stretches=true, inner sep=0pt, minimum width=0.75cm, minimum height=4mm, yshift=-0.0mm, tikzit shape=rectangle]
\tikzstyle{decompwide}=[fill=white, draw, shape=isosceles triangle, shape border rotate=-90, isosceles triangle stretches=true, inner sep=0pt, minimum width=1.5cm, minimum height=4mm, yshift=-0.0mm, tikzit shape=rectangle, tikzit fill=red]
\tikzstyle{decompflip}=[fill=white, draw, shape=isosceles triangle, shape border rotate=90, isosceles triangle stretches=true, inner sep=0pt, minimum width=0.75cm, minimum height=4mm, yshift=-0.0mm, tikzit shape=rectangle]
\tikzstyle{decompwideflip}=[fill=white, draw, shape=isosceles triangle, shape border rotate=90, isosceles triangle stretches=true, inner sep=0pt, minimum width=1.5cm, minimum height=4mm, yshift=-0.0mm, tikzit shape=rectangle, tikzit fill=red]
\tikzstyle{sharpstate}=[point]
\tikzstyle{sharpeffect}=[copoint]
\tikzstyle{green dot}=[fill=white, draw={rgb,255: red,34; green,139; blue,34}, shape=circle, inner sep=.4ex]
\tikzstyle{greensharpstate}=[point, fill=white, draw={rgb,255: red,34; green,139; blue,34}]
\tikzstyle{grbox}=[fill=white, draw={rgb,255: red,34; green,139; blue,34}, shape=rectangle]
\tikzstyle{split node}=[fill=white, draw={rgb,255: red,255; green,155; blue,255}, shape=circle, dashed]

\tikzstyle{dir}=[->]
\tikzstyle{dashline}=[-, loosely dotted, thick]
\tikzstyle{triangle}=[->, style={{-{Triangle[open]}}}]
\tikzstyle{norm}=[-, dashed, color=blue, draw={rgb,255: red,51; green,0; blue,255}]
\tikzstyle{maptostyle}=[{|->}]
\tikzstyle{dashdir}=[->, dashed]
\tikzstyle{thickline}=[-, thick]
\tikzstyle{FCMBox}=[-, draw={rgb,255: red,191; green,191; blue,191}, thick]
\tikzstyle{green line}=[-, draw={rgb,255: red,34; green,139; blue,34}]
\tikzstyle{logbox}=[-, draw={rgb,255: red,1; green,162; blue,12}]
\tikzstyle{highlight}=[-, dashed=]

%% file: figures/box.tikz
\begin{tikzpicture}
	\begin{pgfonlayer}{nodelayer}
		\node [style=map] (0) at (0, 0) {$f$};
		\node [style=none] (1) at (0, 1) {};
		\node [style=none] (2) at (0, -1) {};
		\node [style=label] (3) at (0, 1.5) {$Y$};
		\node [style=label] (4) at (0, -1.5) {$X$};
	\end{pgfonlayer}
	\begin{pgfonlayer}{edgelayer}
		\draw (2.center) to (1.center);
	\end{pgfonlayer}
\end{tikzpicture}

%% file: figures/stateomega.tikz
\begin{tikzpicture}
	\begin{pgfonlayer}{nodelayer}
		\node [style=map] (0) at (0, 0) {$\omega$};
		\node [style=none] (1) at (0, 1) {};
		\node [style=label] (2) at (0, 1.5) {$X$};
	\end{pgfonlayer}
	\begin{pgfonlayer}{edgelayer}
		\draw (1.center) to (0);
	\end{pgfonlayer}
\end{tikzpicture}

%% file: figures/effect.tikz
\begin{tikzpicture}[tikzfig]
	\begin{pgfonlayer}{nodelayer}
		\node [style=map] (0) at (0, 1) {$e$};
		\node [style=none] (1) at (0, 0) {};
		\node [style=label] (2) at (0, -0.5) {$X$};
	\end{pgfonlayer}
	\begin{pgfonlayer}{edgelayer}
		\draw (1.center) to (0);
	\end{pgfonlayer}
\end{tikzpicture}

%% file: figures/scalar.tikz
\begin{tikzpicture}
	\begin{pgfonlayer}{nodelayer}
		\node [style=scalar] (0) at (0, 0) {$r$};
	\end{pgfonlayer}
\end{tikzpicture}

%% file: figures/mat-mor.tikz
\begin{tikzpicture}
	\begin{pgfonlayer}{nodelayer}
		\node [style=none] (0) at (0, 1) {};
		\node [style=none] (1) at (0, -1) {};
		\node [style=map] (2) at (0, 0) {$M$};
		\node [style=label] (3) at (0, 1.5) {$Y$};
		\node [style=label] (4) at (0, -1.5) {$X$};
	\end{pgfonlayer}
	\begin{pgfonlayer}{edgelayer}
		\draw (1.center) to (0.center);
	\end{pgfonlayer}
\end{tikzpicture}

%% file: figures/state.tikz
\begin{tikzpicture}
	\begin{pgfonlayer}{nodelayer}
		\node [style=map] (0) at (0, -0.25) {$\omega$};
		\node [style=none] (1) at (0, 0.75) {};
		\node [style=label] (2) at (0, 1) {$X$};
	\end{pgfonlayer}
	\begin{pgfonlayer}{edgelayer}
		\draw (1.center) to (0);
	\end{pgfonlayer}
\end{tikzpicture}

%% file: figures/state-norm.tikz
\begin{tikzpicture}
	\begin{pgfonlayer}{nodelayer}
		\node [style=map] (0) at (-0.5, -0.75) {$\omega$};
		\node [style=upground] (1) at (-0.5, 0.75) {};
		\node [style=none] (2) at (0.75, 0) {$=$};
		\node [style=none] (3) at (2, 0) {$1$};
	\end{pgfonlayer}
	\begin{pgfonlayer}{edgelayer}
		\draw (1) to (0);
	\end{pgfonlayer}
\end{tikzpicture}

%% file: figures/state-sum.tikz
\begin{tikzpicture}[tikzfig]
	\begin{pgfonlayer}{nodelayer}
		\node [style=map] (0) at (1.25, -0.75) {$\omega$};
		\node [style=upground] (1) at (1.25, 0.75) {};
		\node [style=label] (2) at (0.5, 0.25) {$X$};
	\end{pgfonlayer}
	\begin{pgfonlayer}{edgelayer}
		\draw (1) to (0);
	\end{pgfonlayer}
\end{tikzpicture}

%% file: figures/m-disc.tikz
\begin{tikzpicture}[tikzfig]
	\begin{pgfonlayer}{nodelayer}
		\node [style=upground] (0) at (0.5, 1.25) {};
		\node [style=none] (1) at (0.5, -1) {};
		\node [style=map] (2) at (0.5, 0) {$M$};
		\node [style=label] (3) at (-0.5, 1) {$Y$};
		\node [style=label] (4) at (0.5, -1.5) {$X$};
	\end{pgfonlayer}
	\begin{pgfonlayer}{edgelayer}
		\draw (1.center) to (0);
	\end{pgfonlayer}
\end{tikzpicture}

%% file: figures/exp-value.tikz
\begin{tikzpicture}
	\begin{pgfonlayer}{nodelayer}
		\node [style=map] (0) at (0, -1) {$\omega$};
		\node [style=map] (1) at (0, 1) {$e$};
		\node [style=label] (2) at (-0.5, 0) {$X$};
	\end{pgfonlayer}
	\begin{pgfonlayer}{edgelayer}
		\draw (1) to (0);
	\end{pgfonlayer}
\end{tikzpicture}

%% file: figures/sharp-1.tikz
\begin{tikzpicture}
	\begin{pgfonlayer}{nodelayer}
		\node [style=sharpeffect] (0) at (-3, 0.75) {$x$};
		\node [style=sharpstate] (1) at (-3, -0.75) {$x$};
		\node [style=none] (2) at (-1.5, 0) {$=$};
		\node [style=none] (4) at (-0.25, 0) {$1$};
	\end{pgfonlayer}
	\begin{pgfonlayer}{edgelayer}
		\draw (1) to (0);
	\end{pgfonlayer}
\end{tikzpicture}

%% file: figures/sharp-state-x.tikz
\begin{tikzpicture}[tikzfig]
	\begin{pgfonlayer}{nodelayer}
		\node [style=sharpstate] (0) at (0.5, -0.75) {$x$};
		\node [style=none] (1) at (0.5, 0.5) {};
		\node [style=label] (2) at (0.5, 1) {$X$};
	\end{pgfonlayer}
	\begin{pgfonlayer}{edgelayer}
		\draw (1.center) to (0);
	\end{pgfonlayer}
\end{tikzpicture}

%% file: figures/mat-values.tikz
\begin{tikzpicture}[tikzfig]
	\begin{pgfonlayer}{nodelayer}
		\node [style=sharpeffect] (0) at (0.5, 1.5) {$y$};
		\node [style=none] (1) at (0.5, 1.5) {};
		\node [style=sharpstate] (2) at (0.5, -1.5) {$x$};
		\node [style=map] (3) at (0.5, 0) {$M$};
	\end{pgfonlayer}
	\begin{pgfonlayer}{edgelayer}
		\draw (2) to (1.center);
	\end{pgfonlayer}
\end{tikzpicture}

%% file: figures/nd-disc.tikz
\begin{tikzpicture}[tikzfig]
	\begin{pgfonlayer}{nodelayer}
		\node [style=upground] (0) at (1.5, 0) {};
		\node [style=none] (1) at (1.5, -0.75) {};
	\end{pgfonlayer}
	\begin{pgfonlayer}{edgelayer}
		\draw (1.center) to (0);
	\end{pgfonlayer}
\end{tikzpicture}

%% file: figures/obs.tikz
\begin{tikzpicture}
	\begin{pgfonlayer}{nodelayer}
		\node [style=map] (0) at (0, -0.75) {$\obs$};
		\node [style=none] (2) at (0, -0.5) {};
		\node [style=label] (6) at (0, 1) {$O$};
		\node [style=none] (7) at (0, 0.5) {};
	\end{pgfonlayer}
	\begin{pgfonlayer}{edgelayer}
		\draw (7.center) to (2.center);
	\end{pgfonlayer}
\end{tikzpicture}

%% file: figures/posterior.tikz
\begin{tikzpicture}
	\begin{pgfonlayer}{nodelayer}
		\node [style=map] (0) at (0, -0.75) {$\Mspost$};
		\node [style=none] (2) at (0, -0.5) {};
		\node [style=label] (6) at (0, 1) {$S$};
		\node [style=none] (7) at (0, 0.5) {};
	\end{pgfonlayer}
	\begin{pgfonlayer}{edgelayer}
		\draw (7.center) to (2.center);
	\end{pgfonlayer}
\end{tikzpicture}

%% file: figures/KL.tikz
\begin{tikzpicture}
	\begin{pgfonlayer}{nodelayer}
		\node [style=map] (0) at (0, 0) {$c$};
		\node [style=none] (1) at (0, 1.25) {};
		\node [style=map] (2) at (0, -2) {$\jup{M}{\obs}$};
		\node [style=label] (3) at (-0.5, -1) {$S$};
		\node [style=label] (4) at (0, 1.75) {$O$};
	\end{pgfonlayer}
	\begin{pgfonlayer}{edgelayer}
		\draw (2) to (1.center);
	\end{pgfonlayer}
\end{tikzpicture}

%% file: figures/omegastatealt.tikz
\begin{tikzpicture}
	\begin{pgfonlayer}{nodelayer}
		\node [style=map] (0) at (0, -0.25) {$\omega$};
		\node [style=none] (1) at (0, 0.75) {};
	\end{pgfonlayer}
	\begin{pgfonlayer}{edgelayer}
		\draw (1.center) to (0);
	\end{pgfonlayer}
\end{tikzpicture}

%% file: figures/sigmastatealt.tikz
\begin{tikzpicture}
	\begin{pgfonlayer}{nodelayer}
		\node [style=map] (0) at (0, -0.25) {$\sigma$};
		\node [style=none] (1) at (0, 0.75) {};
	\end{pgfonlayer}
	\begin{pgfonlayer}{edgelayer}
		\draw (1.center) to (0);
	\end{pgfonlayer}
\end{tikzpicture}

%% file: figures/QS-2.tikz
\begin{tikzpicture}
	\begin{pgfonlayer}{nodelayer}
		\node [style=map] (0) at (0.5, -0.5) {$Q$};
		\node [style=none] (1) at (0.5, 0.5) {};
		\node [style=label] (2) at (0.5, 1) {$S$};
	\end{pgfonlayer}
	\begin{pgfonlayer}{edgelayer}
		\draw (1.center) to (0);
	\end{pgfonlayer}
\end{tikzpicture}

%% file: figures/Q-o.tikz
\begin{tikzpicture}
	\begin{pgfonlayer}{nodelayer}
		\node [style=map] (0) at (0, 0.25) {$Q$};
		\node [style=sharpstate] (1) at (0, -1) {$o$};
		\node [style=none] (2) at (0, 1.25) {};
		\node [style=label] (3) at (0, 1.75) {$S$};
	\end{pgfonlayer}
	\begin{pgfonlayer}{edgelayer}
		\draw (2.center) to (1);
	\end{pgfonlayer}
\end{tikzpicture}

%% file: figures/M-o.tikz
\begin{tikzpicture}
	\begin{pgfonlayer}{nodelayer}
		\node [style=map] (0) at (0, 0.25) {$M$};
		\node [style=sharpstate] (1) at (0, -1) {$o$};
		\node [style=none] (2) at (0, 1.25) {};
		\node [style=label] (3) at (0, 1.75) {$S$};
	\end{pgfonlayer}
	\begin{pgfonlayer}{edgelayer}
		\draw (2.center) to (1);
	\end{pgfonlayer}
\end{tikzpicture}

%% file: figures/QO3.tikz
\begin{tikzpicture}
	\begin{pgfonlayer}{nodelayer}
		\node [style=map] (0) at (0, -0.5) {$Q$};
		\node [style=none] (2) at (0, 0.5) {};
		\node [style=label] (3) at (0, 1) {$O$};
	\end{pgfonlayer}
	\begin{pgfonlayer}{edgelayer}
		\draw (2.center) to (0);
	\end{pgfonlayer}
\end{tikzpicture}

%% file: figures/MO3.tikz
\begin{tikzpicture}
	\begin{pgfonlayer}{nodelayer}
		\node [style=map] (0) at (0.5, -0.5) {$M$};
		\node [style=none] (1) at (0.5, 0.5) {};
		\node [style=label] (2) at (0.5, 1) {$O$};
	\end{pgfonlayer}
	\begin{pgfonlayer}{edgelayer}
		\draw (1.center) to (0);
	\end{pgfonlayer}
\end{tikzpicture}

%% file: figures/q.tikz
\begin{tikzpicture}
	\begin{pgfonlayer}{nodelayer}
		\node [style=map] (0) at (0, -0.5) {$q$};
		\node [style=none] (1) at (0, 0.75) {};
		\node [style=label] (2) at (0, 1.25) {$S$};
	\end{pgfonlayer}
	\begin{pgfonlayer}{edgelayer}
		\draw (1.center) to (0);
	\end{pgfonlayer}
\end{tikzpicture}

%% file: figures/o.tikz
\begin{tikzpicture}
	\begin{pgfonlayer}{nodelayer}
		\node [style=map] (0) at (0, -0.5) {$\obs$};
		\node [style=none] (1) at (0, 0.75) {};
		\node [style=label] (2) at (0, 1.25) {$O$};
	\end{pgfonlayer}
	\begin{pgfonlayer}{edgelayer}
		\draw (1.center) to (0);
	\end{pgfonlayer}
\end{tikzpicture}

%% file: figures/M-od.tikz
\begin{tikzpicture}
	\begin{pgfonlayer}{nodelayer}
		\node [style=map] (0) at (0, 0) {$M$};
		\node [style=sharpstate] (1) at (0, -1.25) {$o$};
		\node [style=none] (2) at (0, 1) {};
		\node [style=label] (3) at (0, 1.5) {$S$};
	\end{pgfonlayer}
	\begin{pgfonlayer}{edgelayer}
		\draw (2.center) to (1);
	\end{pgfonlayer}
\end{tikzpicture}

%% file: figures/o-state-small.tikz
\begin{tikzpicture}
	\begin{pgfonlayer}{nodelayer}
		\node [style=map] (0) at (0, -0.5) {$\obs$};
		\node [style=none] (1) at (0, 0.5) {};
		\node [style=label] (2) at (0, 1) {$O$};
	\end{pgfonlayer}
	\begin{pgfonlayer}{edgelayer}
		\draw (1.center) to (0);
	\end{pgfonlayer}
\end{tikzpicture}

%% file: figures/M-o2.tikz
\begin{tikzpicture}
	\begin{pgfonlayer}{nodelayer}
		\node [style=map] (0) at (0, 0) {$M$};
		\node [style=map] (1) at (0, -1.25) {$\obs$};
		\node [style=none] (2) at (0, 1) {};
		\node [style=label] (3) at (0, 1.5) {$S$};
	\end{pgfonlayer}
	\begin{pgfonlayer}{edgelayer}
		\draw (2.center) to (1);
	\end{pgfonlayer}
\end{tikzpicture}

%% file: figures/validity.tikz
\begin{tikzpicture}
	\begin{pgfonlayer}{nodelayer}
		\node [style=map] (0) at (0, -1) {$\omega$};
		\node [style=map] (1) at (0, 1) {$\sigma$};
		\node [style=label] (2) at (0.75, 0) {$X$};
	\end{pgfonlayer}
	\begin{pgfonlayer}{edgelayer}
		\draw (1) to (0);
	\end{pgfonlayer}
\end{tikzpicture}

%% file: figures/vfeup.tikz
\begin{tikzpicture}
	\begin{pgfonlayer}{nodelayer}
		\node [style=map] (0) at (0, -0.25) {$\vup{M}{\obs}$};
		\node [style=none] (1) at (0, 1) {};
		\node [style=label] (2) at (0, 1.5) {$S$};
	\end{pgfonlayer}
	\begin{pgfonlayer}{edgelayer}
		\draw (1.center) to (0);
	\end{pgfonlayer}
\end{tikzpicture}

%% file: figures/Cstate2f.tikz
\begin{tikzpicture}
	\begin{pgfonlayer}{nodelayer}
		\node [style=map] (0) at (0, -0.5) {$C$};
		\node [style=none] (1) at (0, 0.5) {};
		\node [style=label] (2) at (0, 1) {$\EFO$};
	\end{pgfonlayer}
	\begin{pgfonlayer}{edgelayer}
		\draw (1.center) to (0);
	\end{pgfonlayer}
\end{tikzpicture}

%% file: figures/Ms.tikz
\begin{tikzpicture}
	\begin{pgfonlayer}{nodelayer}
		\node [style=map] (0) at (0, 0.25) {$M$};
		\node [style=map] (1) at (0, -1.25) {$s$};
		\node [style=none] (2) at (0, 1.25) {};
		\node [style=label] (3) at (0, 1.75) {$O$};
	\end{pgfonlayer}
	\begin{pgfonlayer}{edgelayer}
		\draw (2.center) to (1);
	\end{pgfonlayer}
\end{tikzpicture}

%% file: figures/MO.tikz
\begin{tikzpicture}
	\begin{pgfonlayer}{nodelayer}
		\node [style=map] (0) at (0, -0.5) {$M$};
		\node [style=none] (1) at (0, 0.5) {};
		\node [style=label] (2) at (0, 1) {$O$};
	\end{pgfonlayer}
	\begin{pgfonlayer}{edgelayer}
		\draw (1.center) to (0);
	\end{pgfonlayer}
\end{tikzpicture}

%% file: figures/C5.tikz
\begin{tikzpicture}
	\begin{pgfonlayer}{nodelayer}
		\node [style=map] (0) at (0.5, -0.5) {$C$};
		\node [style=none] (2) at (0.5, 0.5) {};
		\node [style=label] (3) at (0.5, 1) {$O$};
	\end{pgfonlayer}
	\begin{pgfonlayer}{edgelayer}
		\draw (0) to (2.center);
	\end{pgfonlayer}
\end{tikzpicture}

%% file: figures/cx.tikz
\begin{tikzpicture}
	\begin{pgfonlayer}{nodelayer}
		\node [style=map] (0) at (0, 0.5) {$c$};
		\node [style=none] (1) at (0, 1.5) {};
		\node [style=none] (2) at (0, -0.5) {};
		\node [style=sharpstate] (3) at (0, -0.75) {$x$};
	\end{pgfonlayer}
	\begin{pgfonlayer}{edgelayer}
		\draw (2.center) to (1.center);
	\end{pgfonlayer}
\end{tikzpicture}

%% file: figures/comega.tikz
\begin{tikzpicture}
	\begin{pgfonlayer}{nodelayer}
		\node [style=map] (0) at (0.5, 0.5) {$c$};
		\node [style=none] (1) at (0.5, 1.5) {};
		\node [style=none] (2) at (0.5, -0.75) {};
		\node [style=map] (3) at (0.5, -1) {$\omega$};
	\end{pgfonlayer}
	\begin{pgfonlayer}{edgelayer}
		\draw (2.center) to (1.center);
	\end{pgfonlayer}
\end{tikzpicture}

%% file: figures/validity3alt.tikz
\begin{tikzpicture}
	\begin{pgfonlayer}{nodelayer}
		\node [style=map] (1) at (-0.25, 0.75) {$\prefs$};
		\node [style=label] (2) at (0.5, -0.25) {$\EFO$};
		\node [style=none] (3) at (-0.25, -0.75) {};
		\node [style=map] (4) at (-0.25, -1) {$M$};
	\end{pgfonlayer}
	\begin{pgfonlayer}{edgelayer}
		\draw (3.center) to (1);
	\end{pgfonlayer}
\end{tikzpicture}

%% file: figures/qpii.tikz
\begin{tikzpicture}
	\begin{pgfonlayer}{nodelayer}
		\node [style=map] (0) at (0, -0.5) {$q(\pi)$};
		\node [style=none] (1) at (0, 0.75) {};
		\node [style=label] (2) at (0, 1.25) {$S$};
	\end{pgfonlayer}
	\begin{pgfonlayer}{edgelayer}
		\draw (1.center) to (0);
	\end{pgfonlayer}
\end{tikzpicture}

%% file: figures/prefF.tikz
\begin{tikzpicture}
	\begin{pgfonlayer}{nodelayer}
		\node [style=map] (0) at (0, 0) {$C$};
		\node [style=none] (1) at (0, 1) {};
		\node [style=label] (2) at (0, 1.5) {$F$};
	\end{pgfonlayer}
	\begin{pgfonlayer}{edgelayer}
		\draw (1.center) to (0);
	\end{pgfonlayer}
\end{tikzpicture}

%% file: figures/obs3.tikz
\begin{tikzpicture}[tikzfig]
	\begin{pgfonlayer}{nodelayer}
		\node [style=map] (0) at (0.5, 0) {$\obs$};
		\node [style=none] (1) at (0.5, 0.25) {};
		\node [style=label] (2) at (0.5, 1.75) {$O$};
		\node [style=none] (3) at (0.5, 1.25) {};
	\end{pgfonlayer}
	\begin{pgfonlayer}{edgelayer}
		\draw (3.center) to (1.center);
	\end{pgfonlayer}
\end{tikzpicture}

%% file: figures/obso2.tikz
\begin{tikzpicture}[tikzfig]
	\begin{pgfonlayer}{nodelayer}
		\node [style=none] (0) at (2, 0) {};
		\node [style=map] (1) at (2, 0) {$\obs$};
		\node [style=label] (2) at (2, 1.75) {$O_2$};
		\node [style=none] (3) at (2, 0.5) {};
		\node [style=none] (4) at (2, 1.25) {};
	\end{pgfonlayer}
	\begin{pgfonlayer}{edgelayer}
		\draw (3.center) to (4.center);
	\end{pgfonlayer}
\end{tikzpicture}

%% file: figures/obso1.tikz
\begin{tikzpicture}[tikzfig]
	\begin{pgfonlayer}{nodelayer}
		\node [style=none] (0) at (2, 0) {};
		\node [style=map] (1) at (2, 0) {$\obs_1$};
		\node [style=label] (2) at (2, 1.75) {$O_1$};
		\node [style=none] (3) at (2, 0.5) {};
		\node [style=none] (4) at (2, 1.25) {};
	\end{pgfonlayer}
	\begin{pgfonlayer}{edgelayer}
		\draw (3.center) to (4.center);
	\end{pgfonlayer}
\end{tikzpicture}

%% file: figures/o2.tikz
\begin{tikzpicture}
	\begin{pgfonlayer}{nodelayer}
		\node [style=none] (5) at (3.25, 0) {};
		\node [style=map] (6) at (3.25, 0) {$\obs_2$};
		\node [style=label] (7) at (3.25, 2) {$O_2$};
		\node [style=none] (8) at (3.25, 0.5) {};
		\node [style=none] (9) at (3.25, 1.5) {};
	\end{pgfonlayer}
	\begin{pgfonlayer}{edgelayer}
		\draw (8.center) to (9.center);
	\end{pgfonlayer}
\end{tikzpicture}